\documentclass[10pt]{amsart}

\usepackage{amsmath}
\usepackage{amssymb}
\usepackage{bm}
\usepackage{graphicx}
\usepackage{psfrag}
\usepackage{soul}
\usepackage{color}
\usepackage{hyperref}
\hypersetup{colorlinks=true, linkcolor=blue, citecolor=magenta, urlcolor=wine}
\usepackage{url}
\usepackage{algpseudocode}
\usepackage{fancyhdr}
\usepackage{mathtools}
\usepackage{tikz-cd}
\usepackage{xy}
\input xy
\xyoption{all}
\usepackage{stmaryrd}
\usepackage{calrsfs}

\newtheorem*{maintheorem*}{Main Theorem}
\newtheorem{theorem}{Theorem}[section]
\newtheorem{prop}[theorem]{Proposition}

\newtheorem{lem}[theorem]{Lemma}
\newtheorem{cor}[theorem]{Corollary}

\theoremstyle{definition}
\newtheorem{defin}[theorem]{Definition}

\newtheorem{ex}[theorem]{Example}

\newtheorem{question}[theorem]{Question}
\numberwithin{equation}{section}

\newcommand{\cc}{\mathbb{C}}

\newcommand{\nn}{\mathbb{N}}
\newcommand{\pp}{\mathbb{P}}
\newcommand{\qq}{\mathbb{Q}}
\newcommand{\rr}{\mathbb{R}}
\newcommand{\zz}{\mathbb{Z}}

\providecommand\ldb{\llbracket}
\providecommand\rdb{\rrbracket}

\newcommand{\gp}{\text{gp}}

\newcommand{\uu}{\mathcal{U}}



\begin{document}

\mbox{}
\title{On three families of dense Puiseux monoids}

\author{Scott T. Chapman}
\address{Mathematics Department\\Sam Houston State University\\Huntsville, TX 77340}
\email{scott.chapman@shsu.edu}

\author{Felix Gotti}
\address{Mathematics Department\\MIT\\Cambridge, MA 02139}
\email{fgotti@mit.edu}

\author{Marly Gotti}
\address{Apple Inc.\\One Apple Park Way\\Cupertino, CA 95014}
\email{marlygotti@apple.com}

\author{Harold Polo}
\address{Mathematics Department\\UC Irvine\\Irvine, CA 92697}
\email{harold.polo@uci.edu}

\date{\today}
\maketitle

\begin{abstract}
    A positive monoid is a submonoid of the nonnegative cone of a linearly ordered abelian group. The positive monoids of rank $1$ are called Puiseux monoids, and their atomicity, arithmetic of length, and factorization have been systematically investigated for about ten years. Each Puiseux monoid can be realized as an additive submonoid of the nonnegative cone of $\mathbb{Q}$. We say that a Puiseux monoid is dense if it is isomorphic to an additive submonoid of $\mathbb{Q}_{\ge 0}$ that is dense in $\mathbb{R}_{\ge 0}$ with respect to the Euclidean topology. Every non-dense Puiseux monoid is known to be a bounded factorization monoid. However, the atomic structure as well as the arithmetic and factorization properties of dense Puiseux monoids turn out to be quite interesting. In this paper, we study the atomic structure and some arithmetic and factorization aspects of three families of dense Puiseux monoids. 
\end{abstract}
	
\bigskip
\section{Introduction}
\label{sec:intro}

A submonoid of the nonnegative cone of a linearly ordered abelian group is often called a positive monoid. In particular, rank-one positive monoids are known as Puiseux monoids, a term introduced when their atomic structure was first studied by the second and third authors in~\cite{fG17,GG18}. Since then, not only the atomicity but also the factorization behavior and the arithmetic of lengths of Puiseux monoids have been actively studied by several authors over the last ten years (see the recent papers~\cite{CJMM25} and references therein). As the Grothendieck group of any Puiseux monoid is a rank-one torsion-free abelian group, every Puiseux monoid is isomorphic to a monoid of consisting of rationals under the standard addition \cite[Section~24]{lF70} and, as a result, every nontrivial Puiseux monoid is isomorphism to a monoid consisting of nonnegative rationals under the standard addition \cite[Theorem~2.9]{rG84}. The interested reader can find a survey on the atomicity of Puiseux monoids in~\cite{CGG21} by the first three authors.
\smallskip

Albeit a natural generalization of numerical monoids, Puiseux monoids exhibit a complex and interesting atomic/arithmetic structure. For instance, for any prescribed nonnegative integer $n$, there is a non-atomic Puiseux monoid with exactly $n$ atoms \cite[Proposition~5.4]{fG17}, while there are non-atomic Puiseux monoids with infinitely many atoms \cite[Example~3.5]{fG17}. In addition, there are atomic Puiseux monoids having full systems of sets of length, as proved in~\cite{fG19a} by the second author (the elasticity of Puiseux monoids was studied in~\cite{GO20} by O'Neill and the second author). Moreover, as Theorem~\ref{thm:existence of atomically dense PM} states, there are atomic Puiseux monoids whose sets of atoms are dense in the nonnegative part of the real line under the Euclidean topology.
\smallskip

The remarkable variety of unexpected atomic and factorization behaviors exhibited by Puiseux monoids have provided a rich source of (counter)examples, which have proven instrumental over recent decades in testing the sharpness of theorems and disproving conjectures within commutative algebra and factorization theory. For instance, a Puiseux monoid is the most relevant ingredient in Grams' construction of the first known atomic domain not satisfying the ACCP (see~\cite{aG74}), while Puiseux monoids have been essential in fully addressing Gilmer's question~\cite[page~189]{rG84} about the ascent of atomicity to monoid algebras (see~\cite[Section~5]{CG19} and~\cite{GR25}). Also, in the recent paper~\cite{GGP25}, Gonzalez, Panpaliya, and the second author generalized the class of Puiseux monoids in~\cite[Section~5]{CG19} to argue that neither quasi-atomicity nor almost atomicity (two weaker notions of atomicity introduced by Boynton and Coykendall back in 2015) ascend to monoid algebras.
\smallskip

Let $M$ be a Puiseux monoid, and assume that $M$ consists of nonnegative rationals. Since addition is continuous in the real line under the Euclidean topology, the subspace topology inherited by $M$ is intrinsically linked to its algebraic structure. In particular, if one aims to study the atomic decomposition of elements of $M$, examining neighborhoods of $0$ might provide substantial insight. Indeed, if $0$ is not a limit point of $M$, then one can readily show that $M$ is atomic~\cite[Proposition~4.5]{fG19}. However, when $0$ is a limit point of $M$, the atomic structure of $M$ might become considerably intricate and significantly different for distinct classes of Puiseux monoids. When $0$ is a limit point of $M$, the additive closedness of $M$ implies that $M$ is dense in the nonnegative part of the real line under the Euclidean topology. Thus, $M$ is dense in the nonnegative part of the real line if and only if $0$ is a limit point of $M \setminus \{0\}$. This observation motivates the following definition.

\begin{defin}
	A Puiseux monoid is \emph{dense} if it is dense in the nonnegative part of the real line under the Euclidean topology.
\end{defin}
\noindent Dense Puiseux monoids serve as valuable tools to differentiate between classes in the nested diagram of atomicity introduced by Anderson, Anderson, and Zafrullah in their landmark paper~\cite{AAZ90}, where they presented the bounded and finite factorization properties as two weaker variants of the unique factorization property. Additionally, dense Puiseux monoids help distinguish the nested classes of AP-ness investigated by Anderson and Quintero in~\cite{AQ97}. We dedicate this paper to examining some atomic and factorization aspects of three families of dense Puiseux monoids, which have been selected based on significant examples from recent literature (see \cite{CGG20,GG25,GR25,hP20}). The class of dense Puiseux monoids has also been studied by Bras-Amor\'os and the third author in~\cite{BG21}.
\smallskip

In Section~\ref{sec:Background and Notation}, which is the background section, we introduce most of the standard notation and terminology as well as the non-standard results we shall be using later.
\smallskip

In Section~\ref{sec:general facts of DPM}, we provide some basic results related to dense Puiseux monoids and their monoid homomorphisms. We conclude the section producing atomic Puiseux monoids whose sets of atoms are dense in $\rr_{\ge 0}$.
\smallskip

In Section~\ref{sec:k-prime Reciprocal PM}, we briefly discuss some examples of dense Puiseux monoids generated by reciprocals of length-$k$ elements of $\nn$ (for any fixed $k \in \nn$). 
\smallskip

In Section~\ref{sec:p-adic PM}, we first study the class of $p$-adic Puiseux monoids for all $p \in \pp$ (here $\pp$ denotes the set of primes): for each $p \in \pp$, a $p$-adic Puiseux monoid is a submonoid of the valuation Puiseux monoid
\[
    \nn_0\bigg[\frac1{p}\bigg] := \bigg\{ f\bigg(\frac1{p}\bigg) : f(x) \in \nn_0[x] \bigg\},
\]
where $\nn_0[x]$ is the semiring of polynomials with nonnegative integer coefficients. We establish necessary and sufficient conditions for $p$-adic Puiseux monoids to be atomic. Then we consider, for each function $f \colon \pp \to \nn_0$ such that $p \nmid f(p)$ for any $p \in \pp$ with $f(p) \neq 0$, the Puiseux monoid
\begin{equation} \label{eq:intro M_f}
    M_f := \sum_{p \in \pp} \frac{f(p)}p \nn_0,
\end{equation}
which is an internal sum of $p$-adic monoids, each of them isomorphic to $\nn_0$. Clearly, $M_f$ is dense if and only if the infimum of $\big\{ \frac{f(p)}p : p \in \pp \big\}$ is~$0$. We prove that elements of $M_f$ can be decomposed as a canonical sum, and then we use such a canonical decomposition to show that $M_f$ satisfies the ascending chain condition on principal ideals (ACCP). Finally, for each finite nonempty subset $P$ consisting of primes, we consider the following internal sum
\begin{equation} \label{eq:intro M_P}
    M_P := \sum_{p \in P} \nn_0\bigg[\frac1p\bigg],
\end{equation}
which is a dense Puiseux monoid. We prove that elements in $M_P$ can also be decomposed as a canonical sum, and then we use this to settle some divisibility properties in the monoids~$M_P$.
\smallskip

In Section~\ref{sec:multiplicatively cyclic PM}, which is the final section of the paper, we continue the study of the underlying additive monoid of the subsemiring $\nn_0[q]$ for any $q \in \qq_{>0}$, which have already been investigated by the first three authors in~\cite{GG18,CGG20}. For each $q \in \qq_{>0}$, observe that $\nn_0[q]$ is the Puiseux monoid generated by the set $\{q^n : n \in \nn_0\}$. As we did with the monoids in~\eqref{eq:intro M_f} and~\eqref{eq:intro M_P}, we establish a canonical sum decomposition for $\nn_0[q]$, and then we use it to prove that for any non-unit fraction rational~$q$, the monoid $\nn_0[q]$ satisfies the length-finite factorization property: it is atomic and, for any pair $(r,\ell) \in \nn_0[q] \times \nn_0$, the element $r$ has only finitely many factorizations of length~$\ell$. In the second part of the section, we investigate the algebraic and atomic structure of Puiseux monoids of the form
\[
    M_Q := \big\langle q^n : (q,n) \in Q \times \nn_0 \big\rangle,
\]
where $Q$ is a finite nonempty set consisting of positive rationals. Observe that $M_Q$ specializes to $\nn_0[q]$ when the set $Q$ is the singleton $\{q\}$. It is clear that the $M_Q$ is dense if and only if $\min Q < 1$. After having established a canonical sum decomposition for the monoids $M_Q$, we conclude the paper considering their atomicity, factorizations, and multiplicative closedness.

\bigskip
\section{Preliminary}
\label{sec:Background and Notation}


\medskip
\subsection{General Notation}

The symbols $\pp$, $\mathbb{N}$, and $\mathbb{N}_0$ denote the sets of standard primes, positive integers, and nonnegative integers, respectively. For any $r \in \rr$ and $S \subseteq \rr$, we denote the set $\{s \in S : s \ge r\}$ by $S_{\ge r}$. For any $m,n \in \zz$ with $m \le n$, we let $\ldb m,n \rdb$ denote the discrete interval from $m$ to $n$. We say that a sequence consisting of real numbers is \emph{finitely supported} if only finitely many terms of the sequence are different from zero, whence whether a real sequence is finitely supported is invariant under finitely many terms. 

For any nonzero rational $q \in \qq^\times$, we call the unique relatively primes $\mathsf{n}(q), \mathsf{d}(q) \in \zz$ such that $q = \mathsf{n}(q)/\mathsf{d}(q)$ and $\mathsf{d}(q) > 0$ the \emph{numerator} and the \emph{denominator} of $q$, respectively. We say that a positive rational is a \emph{unit fraction} if its numerator is $1$. For each set $Q$ consisting of nonzero rationals, we call
\[
    \mathsf{n}(Q) := \{\mathsf{n}(q) : q \in Q\} \quad \text{ and } \quad \mathsf{d}(Q) := \{\mathsf{d}(q) : q \in Q \}
\]
the \emph{numerator set} and \emph{denominator set} of $Q$, respectively. For a prime $p$, the $p$-\emph{adic valuation} on $\qq$ is the map $v_p \colon \qq \to \zz \cup \{\infty\}$ defined as follows: $v_p(0) = \infty$ and $v_p(q) = v_p(\mathsf{n}(q)) - v_p(\mathsf{d}(q))$ for any $q \neq 0$, where for $n \in \nn$ the value $v_p(n)$ is the exponent of the maximal power of $p$ dividing $n$. Thus,
\begin{equation*} \label{eq:p-valuation semiadditivity}
	v_p(q_1 + \dots + q_n) \ge \min\{v_p(q_1), \dots, v_p(q_n) \}
\end{equation*}
for all $q_1, \dots, q_n \in \qq_{> 0}$. 

\medskip
\subsection{Commutative Monoids}

Throughout this paper, we tacitly assume that every monoid mentioned is commutative and cancellative. Also, unless we specify otherwise, each monoid $M$ in this paper is written additively ($+$ and $0$ denote the binary operation and the identity element of $M$, respectively). We let $M^\bullet$ denote the set of nonzero elements of $M$. Also, we let $\uu(M)$ denote the group of invertible elements of $M$, and we say that $M$ is \emph{reduced} if the group $\uu(M)$ is trivial. For any $b,c \in M$, we say that $c$ \emph{divides} $b$ in $M$ and write $c \mid_M b$ if there exists $d \in M$ such that $b = c + d$; in this case we write $c \mid_M b$. The monoid $M$ is called a \emph{valuation monoid} if for all $b,c \in M$ either $b \mid_M c$ or $c \mid_M b$. For any subset $S$ of $M$, we let $\langle S \rangle$ denote the submonoid of $M$ generated by $S$, and we say that $M$ is \emph{finitely generated} if $M = \langle S \rangle$ for some finite subset $S$ of $M$.
\smallskip

Since $M$ is cancellative, it can be minimally embedded into an abelian group $\gp(M)$ and such an abelian group, which is unique up to isomorphism, is called the \emph{Grothendieck group} of $M$. The monoid~$M$ is said to be \emph{torsion-free} if its Grothendieck group is torsion-free. The \emph{rank} of~$M$ is defined to be the rank of its Grothendieck group as a $\zz$-module or, equivalently, the dimension of the $\qq$-vector space $\qq \otimes_\zz \gp(M)$. A monoid is called a \emph{positive} if it is isomorphic to a submonoid of the nonnegative cone of a linearly orderable abelian group.
\smallskip

An element $a \in M \setminus \uu(M)$ is an \emph{atom} if whenever $a = b + c$ for some $b,c \in M$, either $b \in \uu(M)$ or $c \in \uu(M)$. The set of atoms of $M$ is denoted by $\mathcal{A}(M)$. Observe that $\mathcal{A}(M)$ is a subset of any generating set of $M$ when $M$ is reduced. If $\mathcal{A}(M)$ is the empty set, then $M$ is said to be \emph{antimatter} (the term antimatter was coined by Coykendall, Dobbs, and Mullins~\cite{CDM99} in the setting of integral domains). An element $b \in M$ is said to be \emph{atomic} if either $b$ is invertible or $b$ can be written as a sum of finitely many atoms (allowing repetitions). Following Cohn~\cite{pC68}, we say that $M$ is atomic provided that every element of $M$ is atomic. A subset $I$ of $M$ is called an \emph{ideal} of $M$ provided that $I+M \subseteq I$, and an ideal of the form $b+M$ for some $b \in M$ is called a \emph{principal ideal}. The monoid $M$ is said to satisfy the \emph{ascending chain condition on principal ideals} (ACCP) provided that for any sequence $(I_n)_{n \ge 1}$ of principal ideals of $M$ that is increasing under set-inclusion, there exists $m \in \nn$ such that $I_n = I_m$ for all $n \ge m$. It is well known and not hard to show that every monoid that satisfies the ACCP is atomic.
\smallskip

Assume throughout the rest of this section that $M$ is an atomic monoid. The free commutative monoid on $\mathcal{A}(M/\uu(M))$ is denoted by $\mathsf{Z}(M)$, and the elements of $\mathsf{Z}(M)$ are called \emph{factorizations}. Let $\phi \colon \mathsf{Z}(M) \to M/\uu(M)$ be the unique monoid homomorphism fixing $a + \uu(M)$ for all $a \in \mathcal{A}(M)$. We say that $z = a_1 \cdots a_\ell \in \mathsf{Z}(M)$ is a \emph{factorization} of $b \in M$ if $\phi(z) = b + \uu(M)$, in which case $\ell$ is referred to as the \emph{length} of $z$ and is denoted by $|z|$. For each $b \in M$, set
\[
	\mathsf{Z}(b) := \phi^{-1}(b + \uu(M)) \subseteq \mathsf{Z}(M).
\]
If $\mathsf{Z}(b)$ is a singleton for all $b \in M$, then $M$ is called a \emph{unique factorization monoid} (UFM). More generally, if $\mathsf{Z}(b)$ is finite for all $b \in M$, then $M$ is called a \emph{finite factorization monoid} (FFM). Every finitely generated monoid is an FFM \cite[Proposition~2.7.8]{GH06}. For each $b \in M$, set
\[
	\mathsf{L}(b) := \{|z| : z \in \mathsf{Z}(b)\}.
\]
If $\mathsf{L}(b)$ is finite for all $b \in M$, then $M$ is called a \emph{bounded factorization monoid} (BFM). It follows from the definitions that every FFM is a BFM, and it is well known that every BFM satisfies the ACCP \cite[Corollary~1]{fH92}. The system $\{\mathsf{L}(b) : b \in M\}$ has been fairly investigated during the past few years for Puiseux monoids $M$ (see \cite{ACHP07,CGLM11,GS18} for numerical monoids and \cite{GG25,fG19,GO20} for Puiseux monoids). For each pair $(b,\ell) \in M \times \nn_0$, we set
\[
	\mathsf{Z}_\ell(b) := \{ z \in \mathsf{Z}(b) : |z| = \ell \}.
\]
If $\mathsf{Z}_\ell(b)$ is finite for all pairs $(b,\ell) \in M \times \nn_0$, then $M$ is called a \emph{length-finite factorization monoid} (LFFM). It follows from the definitions that a monoid is an FFM if and only if it is both a BFM and an LFFM, whence we can think of the property of being an LFFM and that of being a BFM properties to be complementary with respect to the property of being an FFM. The notion of an LFFM was introduced by Geroldinger and Zhong~\cite{GZ21}, and it was further investigated by Jiang, Kanungo, and Kim~\cite{JKK24} recently.
\smallskip

One of the most elementary families of atomic monoids is the class of numerical monoids. A \emph{numerical monoid} is a co-finite submonoid of the additive monoid~$\nn_0$. Each numerical monoid has a unique minimal generating set, which is finite. Thus, each numerical monoid is an FFM. If $\{a_1, \dots, a_n\}$ is the minimal generating set of a numerical monoid $N$, then $\mathcal{A}(N) = \{a_1, \dots, a_n\}$ and $\gcd(a_1, \dots, a_n) = 1$. Let~$N$ be a numerical monoid that is not $\nn_0$ with minimal generating set $\mathcal{A}(N) = \{a_1, \dots, a_n\}$. The \emph{Frobenius number} of $N$, denoted in this paper by $\mathsf{f}(N)$, is the minimum $n \in \nn$ such that $\zz_{\ge n+1} \subseteq N$. Although a general explicit formula for $\mathsf{f}(N)$ in terms of $\{a_1, \dots, a_n \}$ is unknown, it is not difficult to show that $\mathsf{f}(N) = a_1 a_2 - a_1 - a_2$ when $n=2$, and a formula when $n=3$ has been established by Tripathi in~\cite{aT17}.

\bigskip
\section{Density of Puiseux Monoids}
\label{sec:general facts of DPM}

In the first part of this section, we briefly discuss homomorphisms and isomorphisms between Puiseux monoids. As previously noted, the density of a Puiseux monoid is determined by its behavior in any neighborhood of $0$, or equivalently, around any generating set. These and other related matters are also discussed in this section.

\medskip
\subsection{Homomorphisms and Isomorphisms}

Since the standard operation of addition is continuous with respect to the Euclidean topology, the additive operations of Puiseux monoids and their homomorphisms are also continuous. This significantly restricts the set of homomorphisms between Puiseux monoids. In this first section, we describe the homomorphisms between Puiseux monoids, and then we propose an isomorphism criterion we will use later.

\begin{lem} \label{lem:isomorphism between PM}
	The homomorphisms between two Puiseux monoids are precisely those given by rational multiplication.
\end{lem}

\begin{proof}
	Let $M$ and $M'$ be two Puiseux monoids, and let $\alpha \colon M \to M'$ be a homomorphism. If $\alpha$ is the trivial homomorphism, then it is multiplication by $0$. Therefore let us assume that $\alpha$ is not the trivial homomorphism, which implies that $M \neq \{0\}$. As $N = M \cap \nn_0$ is a nontrivial submonoid of the additive monoid of nonnegative integers, it has a nonempty minimal set of generators, namely $\{s_1, \dots, s_k\}$. Because $\alpha$ is not the zero homomorphism, there exists $j \in \ldb 1,k \rdb$ such that $\alpha(s_j) \neq 0$. Set $a := \frac{\alpha(s_j)}{s_j}$. For $q \in M^\bullet$ and $n_1, \dots, n_k \in \nn_0$ such that $\mathsf{n}(q) = n_1 s_1 + \dots + n_k s_k$, the fact that $s_i \alpha(s_j) = \alpha(s_i s_j) = s_j \alpha(s_i)$ for each $i \in \ldb 1,k \rdb$ implies that
	\[
		\alpha(q) = \frac 1{\mathsf{d}(q)} \alpha(\mathsf{n}(q)) = \frac 1{\mathsf{d}(q)} \sum_{i=1}^k n_i \alpha (s_i) = \frac 1{\mathsf{d}(q)} \sum_{i=1}^k n_i s_i \frac{\alpha (s_j)}{s_j} = qa.
	\]
	Hence the homomorphism $\alpha$ is multiplication by the rational $a$. On the other hand, it follows immediately that, for all $r \in \qq_{> 0}$, the map $M \to M'$ defined by $q \mapsto rq$ is a homomorphism as long as $rM \subseteq M'$.
\end{proof}

The rest of this section is dedicated to establishing the existence of infinitely many non-isomorphic Puiseux monoids whose sets of atoms are dense in $\rr_{\ge 0}$. If two Puiseux monoids $M$ and $M'$ are isomorphic, we write $M \cong M'$. It follows, from the next lemma, that two Puiseux monoids are isomorphic if and only if they are rational multiples of each other.

\begin{lem} \label{lem:non-isomorphic PM}
	Let $P$ and $Q$ be infinite sets of primes in $\mathbb{P}$ that are disjoint, and let $M_P := \langle a_p : p \in P \rangle$ and $M_Q := \langle b_q : q \in Q \rangle$ be Puiseux monoids such that for all $p \in P$ and $q \in Q$, $\mathsf{d}(a_p)$ and $\mathsf{d}(b_q)$ are powers of $p$ and $q$, respectively. Then $M_P \ncong M_Q$.
\end{lem}

\begin{proof}
	Suppose, by way of contradiction, that $M_P \cong M_Q$. By Lemma~\ref{lem:isomorphism between PM}, there exists $r \in \qq_{> 0}$ such that $M_Q = r M_P$. If $p$ is a prime in $P$ such that $p \nmid \mathsf{n}(r)$, then $r b_p$ would be an element of $M_P$ such that $\mathsf{d}(r b_p)$ is a power of $p$ and, therefore, $p \in Q$. But this contradicts the fact that $P \cap Q$ is empty.
\end{proof}

\medskip
\subsection{General Facts on the Density of Puiseux Monoids}
As mentioned in the introduction, a Puiseux monoid is dense if and only if it is topologically dense in $\rr_{\ge 0}$ with respect to the inherited standard topology. We proceed to show that a Puiseux monoid $M$ is dense if and only if $0$ is a limit point of $M^\bullet$. 

\begin{prop} \label{prop:characterization of dense PM}
	An additive submonoid of $\qq_{\ge 0}$ is a dense Puiseux monoid if and only if $0$ is a limit point of its subset of nonzero elements.
\end{prop}

\begin{proof}
	The forward implication is immediate. Suppose, conversely, that $M$ is a Puiseux monoid such that $0$ is a limit point of $M^\bullet$. Let $(r_n)_{n \ge 1}$ be a sequence in $M^\bullet$ converging to $0$. Fix any $p \in \rr_{> 0}$. To check that $p$ is a limit point of $M$, fix $\epsilon > 0$. Because $\lim r_n = 0$, there exists $n \in \nn$ such that $r_n < \min\{p, \epsilon\}$. Let $m$ be the maximum integer such that $p - mr_n > 0$, and take $r = mr_n$. By the maximality of $m$,
	\[
		0 < p - r  = p - (m+1) r_n + r_n \le r_n < \epsilon.
	\]
	Thus, for any arbitrary $\epsilon \in \rr_{>0}$, we have found $r \in M \setminus \{p\}$ such that $|p-r|< \epsilon$, whence $p$ is a limit point of $M^\bullet$. We conclude that $M$ is a dense Puiseux monoid.
\end{proof}

The following corollary follows directly from the fact that a Puiseux monoid $M$ is a BFM and so an atomic monoid provided that~$0$ is not a limit point of $M^\bullet$ (see \cite[Proposition~4.5]{fG19}).

\begin{cor}
	Every non-atomic Puiseux monoid is topologically dense in $\rr_{\ge 0}$.
\end{cor}

To generate a dense Puiseux monoid, it suffices to take a sequence of positive rationals having $0$ as a limit point. Let $\mathcal{P} = \{P_n : n \in \nn\}$ be a collection of infinite subsets of primes in $\mathbb{P}$ such that $P_i \cap P_j$ is empty for $i \neq j$. Now for each $j \in \nn$, consider the Puiseux monoid $M_j = \big\langle \frac1p : p \in P_j \big\rangle$. Because every $P_j$ is infinite, each $M_j$ is dense. Moreover, $M_i \ncong M_j$ for $i \neq j$; this is an immediate consequence of Lemma~\ref{lem:non-isomorphic PM}. Hence we can conclude that there are countably many non-isomorphic dense Puiseux monoids. We proceed to produce a class of atomic Puiseux monoids whose sets of atoms are dense in $\rr_{\ge 0}$. 

\begin{theorem} \label{thm:existence of atomically dense PM}
	There are infinitely many non-isomorphic Puiseux monoids whose sets of atoms are dense in $\rr_{\ge 0}$.
\end{theorem}

\begin{proof}
	Let $S$ be the set consisting of all the nonzero elements of the valuation Puiseux monoid $\nn_0\big[\frac1p\big]$, the Puiseux monoid consisting of all nonnegative $p$-adic integers. It is clear that $S$ is dense in $\rr$ is dense in $\rr_{\ge 0}$ for every $p \in \pp$.
    
    Fix $\ell \in \rr_{> 0}$ and let us verify that $\ell$ is a limit point of $S$ in $\rr_{\ge 0}$. For this, fix $\epsilon > 0$ and then take $m,n \in \nn$ such that $\frac1{p^n} < \epsilon$ and $\frac{m}{p^n} < \ell \le \frac{m+1}{p^n}$. It follows immediately $s := \frac{m}{p^n}$ of $S$ satisfies that $\big|\ell - \frac{m}{p^n}| < \epsilon$. Since $\epsilon$ was arbitrarily taken, $\ell$ is a limit point of $S$. It is obvious that $0$ is also a limit point of $S$. Hence $S$ is dense in $\rr_{\ge 0}$.
	
	Now take $(r_n)_{n \ge 1}$ to be a sequence of positive rationals with underlying set $R$ dense in $\rr_{\ge 0}$. Also, consider the collection $\mathcal{P} := \{P_n : n \in \nn\}$ of infinite sets of primes such that $P_i \cap P_j$ is empty for $i \neq j$. For each $j \in \nn$, let $P_j := \{p_{jk} : k \in \nn\}$. Now for every $j \in \nn$ and $p_{jk} \in P_j$ the set
	\[
		\bigg\{\frac m{p_{jk}^n} : m,n \in \nn \text{ and } p \nmid m \bigg\}
	\]
	is dense in $\rr_{\ge 0}$. Therefore, for every natural $k$, there exist $m_k, n_k \in \nn$ such that  the inequality $\big| r_k - \frac{m_k}{p_{jk}^{n_k}} \big| < \frac1k$ holds. Consider the Puiseux monoid
	\begin{equation} \label{eq:PM with dense set of atoms}
		M_j = \bigg \langle \frac{m_k}{p_{jk}^{n_k}} : k \in \nn \bigg \rangle.
	\end{equation}
	Because distinct generators in \eqref{eq:PM with dense set of atoms} have powers of distinct primes in their denominators, it follows that $M_j$ is atomic and $\mathcal{A}(M_j) = \big\{ \frac{m_k}{p_{jk}^{n_k}} : k \in \nn \big\}$. Finally, we are led to verify that $\mathcal{A}(M_j)$ is dense in $\rr_{\ge 0}$. To complete this, take $x \in \rr_{\ge 0}$ and then fix $\epsilon > 0$. Since $R$ is dense in $\rr_{\ge 0}$, there exists $k \in \nn$ large enough such that $\frac1k < \frac{\epsilon}2$ and $|x - r_k| < \frac{\epsilon}2$. Consequently, $\big| r_k - \frac{m_k}{p_{jk}^{n_k}} \big| < \frac1k < \frac{\epsilon}2$, which implies that
	\[
		\Big{|} x - \frac{m_k}{p_{jk}^{n_k}} \Big{|} < |x - r_k| + \Big{|} r_k - \frac{m_k}{p_{jk}^{n_k}} \Big{|} < \epsilon.
	\]
	This means that $\mathcal{A}(M_j)$ is dense in $\rr_{\ge 0}$, as desired. By Lemma~\ref{lem:non-isomorphic PM}, the Puiseux monoids in $\mathcal{P}$ are pairwise non-isomorphic. 
\end{proof}

\bigskip
\section{$k$-Prime Reciprocal Monoids}
\label{sec:k-prime Reciprocal PM}

For a nonempty set $P$ consisting of primes, it seems that the Puiseux monoid $M_P := \langle \frac1{p} : p \in P\big\rangle$ was first considered in the setting of commutative ring theory~\cite[Example~2.1]{AAZ90} (for $P = \pp$). In the mentioned example, $M_P$ is the monoid of exponents of the monoid algebra $F[M]$ (where $F$ was taken to be a field), which was proved to be an integral domain that satisfies the ACCP but is not a BFD (see~\cite[Proposition~4.2]{fG22}). Our main purpose in this section is to discuss a natural generalization of the monoids $M_P$ and then provide insight into the atomicity of this generalized family.  We note that another generalization of the prime reciprocal Puiseux monoid was recently considered in~\cite[Proposition~3.1]{GG25} by Geroldinger and the second author in connection with the length-sets of monoid algebras not having the bounded factorization property.
	
\begin{defin}
	For an infinite set $P$ consisting of primes in $\mathbb{P}$ and $k \in \nn$, a \emph{$k$-prime reciprocal} (\emph{Puiseux}) \emph{monoid} over $P$ is a Puiseux monoid generated by a set $S$ such that $\mathsf{d}(s)$ is the product of $k$ distinct primes in $P$ for all $s \in S$.
\end{defin}

For a set $S$ and a nonnegative integer $n$, we let $\binom Sn$ denote the collection of subsets of $S$ of cardinality~$n$. The \emph{elementary} $k$-prime reciprocal monoid over an infinite set of primes $P$ is the monoid
\[
	M_{P,k} = \langle S_k \rangle, \text{ where } S_k = \bigg\{\frac 1{p_1 \cdots p_k} : \{p_1, \dots, p_k\} \in \binom{P}{k} \bigg\}.
\]
We call the elements of $S_k$ the \emph{elementary generators} of $M_{P,k}$. We now show that while the elementary 1-prime reciprocal monoid is atomic, for each $k \ge 2$, every elementary $k$-prime reciprocal monoid is antimatter.

\begin{theorem} \label{thm:k-prime reciprocal monoids}
	Let $P$ be an infinite set consisting of primes and, for $k \in \nn$, consider the elementary $k$-prime reciprocal monoid $M_{P,k}$. Then the following statements hold.
	\begin{enumerate}
		\item If $k=1$, then $M_{P,k}$ satisfies the ACCP and so is atomic.
		\smallskip
		
		\item If $k \ge 2$, then $M_{\pp,k}$ is antimatter.
	\end{enumerate}
\end{theorem}

\begin{proof}
	(1) This is well known (see \cite[Proposition~3.1]{GG25}).
	\smallskip
	
	(2) First, we show that for every natural $N$ and distinct primes $p$ and $q$ there exist $m,n \in \nn$ and $p', q' \in \pp$ with $q' > p' > N$ such that
	\begin{equation} \label{eq:k-primary failing to be atomic}
		p'q' = mqq' + npp' + pq.
	\end{equation}
	Let $p$ and $q$ be such two distinct primes. Since $\gcd(p,q) = 1$, Dirichlet's theorem on arithmetic progressions of primes ensures the existence of $m \in \nn$ such that $p' = mq + p$ is a prime greater than~$N$. Dirichlet's theorem comes into play again to yield a natural $n$ such that $q' = np' + q$ is a prime. Therefore one finds that
	\[
		p' q' = mqq' + pq' = mqq' + p(np' + q) = mqq' + npp' + pq.
	\]
	Now consider the elementary $2$-prime reciprocal monoid $M_2 := M_{\pp,2}$. If $p$ and $q$ are distinct primes, then by the argument given above, there exist $m,n \in \nn$ and $p',q' \in \pp$ satisfying $q' > p' > \max\{p,q\}$ such that the identity \eqref{eq:k-primary failing to be atomic} holds. Dividing both sides of \eqref{eq:k-primary failing to be atomic} by $pqp'q'$, we obtain
	\[
		\frac{1}{pq} = m\frac{1}{pp'} + n\frac{1}{qq'} + \frac{1}{p'q'}.
	\]
	As a consequence, no element of $M_2$ can be an atom, which means that $M_2$ is an antimatter Puiseux monoid.
	
	At this point we are in a position to check the more general fact that, for each $k \ge 2$, the elementary $k$-prime reciprocal monoid $M_k \coloneqq M_{\mathbb{P},k}$ is antimatter. To do this, fix an arbitrary elementary generator $(p_1 \cdots p_k)^{-1}$ of $M_k$. As before, there exist $m,n \in \nn$ and $p', q' \in \pp$ satisfying that $q' > p' > \max\{p_1, \dots, p_k\}$ and
	\begin{equation} \label{eq:k-primary failing to be atomic 2}
		\frac{1}{p_1p_2} = m\frac{1}{p_1p'} + n\frac{1}{p_2q'} + \frac{1}{p'q'}.
	\end{equation}
	Multiplying both sides of \eqref{eq:k-primary failing to be atomic 2} by $(p_3 \cdots p_k)^{-1}$, one obtains the following equality:
	\[
		\frac{1}{p_1p_2 \cdots p_k} = m\frac{1}{p_1p'p_3 \cdots p_k} + n\frac{1}{p_2q'p_3 \cdots p_k} + \frac{1}{p'q'p_3 \cdots p_k}.
	\]
	Hence no element of $M_k$ can be an atom and, thus, $M_k$ is antimatter.
\end{proof}

The next example sheds some light upon the fact a monoid generated by infinitely many elementary generators of the elementary $k$-prime reciprocal monoid may not be antimatter.

\begin{ex} \label{ex:elementary k-primary PM-like 1}
	Let $(p_n)_{n \ge 1}$ be a sequence whose underlying set is $\pp$. For $k \in \nn$, consider the $k$-prime reciprocal monoid
	\[
		M = \Big\langle a_n := \prod_{i=1}^k\frac 1{p_{nk+i}} : n \in \nn_0 \Big\rangle.
	\]
	Because, for each index $n \in \nn$, the prime $p_{nk+1} \mid \mathsf{d}(a_m)$ if and only if $m = n$, we can use the $p_{nk+1}$-adic valuation map to argue that $a_n$ is an atom of $M$. As a consequence, $M$ is an atomic monoid whose sets of atoms is
    \[
        \mathcal{A}(M) = \big\{a_n : n \in \nn_0 \big\}.
    \]
    Moreover, since the terms of the sequence $(\mathsf{d}(a_n))_{n \ge 0}$ are pairwise relatively primes, the Puiseux monoid $M$ must satisfy the ACCP \cite[Proposition~3.1]{GG25}. Finally, we can see that $M$ does not have the bounded factorization property: for each $n \in \nn$, we can write $1 = \mathsf{d}(a_n)a_n$ (because $\mathsf{n}(a_n) = 1$) and so $\{\mathsf{d}(a_n) : n \in \nn \} \subseteq \mathsf{L}(1)$ (indeed, it is not difficult to verify that this inclusion is an equality).
\end{ex}

As Example~\ref{ex:elementary k-primary PM-like 2} illustrates, a $k$-prime reciprocal monoid generated by multiples of the elementary generators of $M_{\pp,k}$ may be atomic.

\begin{ex} \label{ex:elementary k-primary PM-like 2}
	Let $(p_n)_{n \ge 1}$ be a sequence whose underlying set is $\pp$. For $k \in \nn$, consider the following $k$-prime reciprocal monoid:
	\[
		M = \Big\langle a_S := \sum_{s \in S} \frac 1{p_s} : S \in \binom{\nn}{k} \Big\rangle.
	\]
	Observe that $M$ is a submonoid of $M_{\pp,k}$. For each subset $S \in \binom{\nn}k$, the element $\frac1{p_s}$ belongs to the prime reciprocal monoid $M_\pp = \big\langle \frac1p : p \in \pp \big\rangle$ for every $s \in S$, whence $a_S = \sum_{s \in S} \frac1{p_s} \in M_\pp$. Therefore $M$ is a submonoid of $M_\pp$. It follows from Theorem~\ref{thm:k-prime reciprocal monoids} that $M_\pp$ satisfies the ACCP, and so the fact that $M_\pp$ is a reduced monoid ensures that every submonoid of $M_\pp$ also satisfies the ACCP. In particular, $M$ satisfies the ACCP and must be atomic.
\end{ex}

\bigskip
\section{$p$-adic Puiseux Monoids and their Internal Sums}
\label{sec:p-adic PM}

For a prime $p$, the Puiseux monoid $\nn_0\big[ \frac1p\big]$ is clearly  an antimatter valuation monoid.  The submonoids of the Puiseux monoid $\nn_0\big[ \frac1p\big]$ are the central algebraic objects of this section. Let us introduce some convenient terminology.

\begin{defin}
	Let $p$ be a prime. We say that a Puiseux monoid $M$ is \emph{$p$-adic} if $M$ is a submonoid of $\nn_0\big[\frac1p\big]$ or, equivalently, $v_q(M) \subseteq \mathbb{N}_0$ for every $q \in \pp \setminus \{p\}$.
\end{defin}

We use the term $p$-\emph{adic monoid} as a shorthand for $p$-adic Puiseux monoid. Throughout this section, each time we define a $p$-adic monoid by specifying a sequence of generators $(r_n)_{n \ge 1}$, we shall implicitly assume that $(\mathsf{d}(r_n))_{n \ge 1}$ diverges to infinity; this assumption comes without loss of generality because in order to generate a Puiseux monoid, we only need to repeat each denominator finitely many times. On the other hand, $\lim \mathsf{d}(r_n) = \infty$ does not affect the generality of the results we prove in this section.  To see this, if $(\mathsf{d}(r_n))_{n \ge 1}$ is a bounded sequence, then the $p$-adic monoid generated by $(r_n)_{n \ge 1}$ is finitely generated and, therefore, isomorphic to a numerical monoid.

\medskip
\subsection{Atomicity of $p$-Adic Monoids}

Although $\nn_0\big[ \frac1p \big]$ is antimatter, it contains plenty of submonoids with diverse and complex atomic behavior. In this section, we delve into the atomicity of $p$-adic monoids.

Strongly bounded Puiseux monoids were considered in~\cite{GG18}. A set $Q$ consisting of rationals is called \emph{strongly bounded} if $\mathsf{n}(Q)$ is bounded, and a Puiseux monoid is called \emph{strongly bounded} if it can be generated by a strongly bounded set of rationals. Strongly bounded $p$-adic monoids happen to have only finitely many atoms, as the next proposition indicates.

\begin{prop} \label{prop:strongly bounded p-adic monoids has finitely many atoms}
	A strongly bounded $p$-adic monoid has only finitely many atoms.
\end{prop}

\begin{proof}
	For $p \in \pp$, let $M$ be a strongly bounded $p$-adic monoid. Let $(r_n)_{n \ge 1}$ be a generating sequence for $M$ with underlying set $R$ satisfying that $\mathsf{n}(R) = \{n_1, \dots, n_k\}$ for some $k, n_1, \dots, n_k \in \nn$. For each $i \in \ldb 1,k \rdb$, set $R_i := \{r_n : \mathsf{n}(r_n) = n_i\}$ and $M_i = \langle R_i \rangle$. The fact that $R \subseteq M_1 \cup \dots \cup M_k$, along with $\mathcal{A}(M) \cap M_i \subseteq \mathcal{A}(M_i)$, implies that
	\[
		\mathcal{A}(M) \subseteq \bigcup_{i=1}^k \mathcal{A}(M_i).
	\]
	Thus, showing that $\mathcal{A}(M)$ is finite amounts to verifying that $|\mathcal{A}(M_i)| < \infty$ for each $i \in \ldb 1,k \rdb$. Fix $i \in \ldb 1,k\rdb$. If $M_i$ is finitely generated, then $|\mathcal{A}(M_i)| < \infty$. Let us assume, therefore, that $M_i$ is not finitely generated. This means that there exists a strictly increasing sequence $(\alpha_n)_{n \ge 1}$ such that $M_i = \big\langle \frac{n_i}{p^{\alpha_n}} : n \in \nn \big\rangle$. Because $\frac{n_i}{p^\alpha_n} = p^{\alpha_{n+1} - \alpha_n}\big( \frac{n_i}{p^{\alpha_{n+1} }}\big)$, the monoid $M_i$ satisfies that $|\mathcal{A}(M_i)| = 0$. Hence we conclude that $\mathcal{A}(M)$ is finite. 
\end{proof}

We are now in a position to give a necessary condition for the atomicity of $p$-adic monoids.

\begin{theorem} \label{thm:necessary condition for the atomicity of p-adic PM}
	For $p \in \pp$, let $M$ be a $p$-adic monoid with $\mathcal{A}(M) = \{r_n : n \in \nn\}$ and $\lim r_n = 0$. If~$M$ is atomic, then $\lim \mathsf{n}(r_n) = \infty$ (i.e., the set $\mathsf{n}(\mathcal{A}(M))$ is not bounded).
\end{theorem}

\begin{proof}
	Set $a_n \coloneqq \mathsf{n}(r_n)$ and $p^{\alpha_n} \coloneqq \mathsf{d}(r_n)$ for every natural $n$. Suppose, by way of contradiction, that $\lim a_n \neq \infty$. Then there exists $m \in \nn$ such that $a_n = m$ for infinitely many $n \in \nn$. For each positive divisor $d$ of $m$ we define the Puiseux monoid
	\[
		M_d := \langle S_d \rangle, \ \text{ where } \ S_d =  \bigg\{ \frac{a_{k_n}}{p^{\alpha_{k_n}}} : a_{k_n} = m \ \text{or} \ \gcd(m, a_{k_n}) = d \bigg\}.
	\]
	Observe that $\mathcal{A}(M)$ is included in the union of the $M_d$. In addition, for each positive divisor $d$ of $m$, the inclusion $\mathcal{A}(M) \cap M_d \subseteq \mathcal{A}(M_d)$ holds and so
	\begin{equation} \label{eq:inclusion of atoms}
		\mathcal{A}(M) \subseteq \bigcup_{d \mid m} \mathcal{A}(M_d).
	\end{equation}
	Because $\mathcal{A}(M)$ contains infinitely many atoms, the inclusion \eqref{eq:inclusion of atoms} implies the existence of a positive divisor~$d$ of $m$ such that $|\mathcal{A}(M_d)| = \infty$. Set $N_d = \frac 1d M_d$. Since $d$ divides $\mathsf{n}(q)$ for all $q \in M_d$, it follows that $N_d$ is also a $p$-adic monoid. In addition, the fact that $N_d$ is isomorphic to $M_d$ implies that $|\mathcal{A}(N_d)| = |\mathcal{A}(M_d)| = \infty$. After setting $b_n := \frac{a_{k_n}}d$ and $\beta_n = \alpha_{k_n}$ for every $n \in \nn$ such that either $a_{k_n} = m$ or $\gcd(m, a_{k_n}) = d$, we see that
	\[
		N_d = \bigg\langle \frac{b_n}{p^{\beta_n}} : n \in \nn \bigg\rangle.
	\]
	As $a_n = m$ for infinitely many $n \in \nn$, the sequence $(\beta_n)_{n \ge 1}$ is an infinite subsequence of $(\alpha_n)_{n \ge 1}$ and, therefore, it diverges to infinity. In addition, as $\lim \frac{a_n}{p^{\alpha_n}} = 0$, it follows that $\lim \frac{b_n}{p^{\beta_n}} = 0$.
	
	Now we argue that $\mathcal{A}(N_d)$ is finite, yielding the desired contradiction. Take $m' = \frac{m}d$. Since there are infinitely many $n \in \nn$ such that $b_n = m'$, it is guaranteed that $\frac{m'}{p^n} \in N_d$ for every $n \in \nn$. In addition, $\gcd(m', b_n) = 1$ for each $b_n \neq m'$. If $b_n \neq m'$ for only finitely many $n$, then $N_d$ is strongly bounded and Proposition~\ref{prop:strongly bounded p-adic monoids has finitely many atoms} ensures that $\mathcal{A}(N_d)$ is finite. Suppose otherwise that $\gcd(b_n, m') = 1$ (i.e., $b_n \neq m'$) for infinitely many $n \in \nn$. For a fixed index $i \in \nn$ with $b_i \neq m'$ take an index $j \in \nn$ such that $\gcd(b_j, m') = 1$ and large enough so that $b_i p^{\beta_j - \beta_i} > b_j m'$ (the existence of such an index $j$ is guaranteed by the fact that $\lim \frac{b_n}{p^{\beta_n}} = 0$). Since $b_i p^{\beta_j - \beta_i} > b_j m' > \mathsf{f}(\langle b_j, m'\rangle)$, we can take $c,c' \in \nn$ such that $b_i p^{\beta_j - \beta_i} = c b_j + c' m'$, which means that
	\[
		\frac{b_i}{p^{\beta_i}} = c \frac{b_j}{p^{\beta_j}} + c' \frac{m'}{p^{\beta_j}}.
	\]
	As $\frac{b_j}{p^{\beta_j}}, \frac{m'}{p^{\beta_j}} \in N_d^\bullet$, it follows that $\frac{b_i}{p^{\beta_i}} \notin \mathcal{A}(N_d)$. Because the index $i$ was arbitrarily taken, $N_d$ is antimatter. In particular, $\mathcal{A}(N_d)$ is finite, which leads to a contradiction.
\end{proof}

With the notation as in the statement of Theorem~\ref{thm:necessary condition for the atomicity of p-adic PM}, the $p$-adic monoid $M$ may fail to be atomic even when the conditions $\lim r_n = 0$ and $\lim \mathsf{n}(r_n) = \infty$ hold simultaneously. The next example sheds some light upon this observation.

\begin{ex}
    For each $n \in \nn_0$, we let $F_n$ denote the $n$-th Fermat number: $F_n := 2^{2^n} + 1$. Then we consider the following Puiseux monoid:
	\begin{equation} \label{eq:p-adic not atomic}
		M := \Big\langle \frac{F_n - 2}{2^{2^{n+1}}}, \frac{F_n}{2^{2^{n+1}}} : n \in \nn \Big\rangle.
	\end{equation}
    Observe that, for each $n \in \nn$,
    \[
		\frac{1}{2^{2^n-1}} = \frac{(2^{2^n} - 1) + \big( 2^{2^n} + 1\big)}{2^{2^{n+1}}} = \frac{F_n - 2}{2^{2^{n+1}}} + \frac{F_n}{2^{2^{n+1}}} \in M.
    \]
    This clearly implies that $\frac1{2^n} \in M$ for every $n \in \nn$. Therefore the monoid $M$ contains the valuation Puiseux monoid $\nn_0\big[\frac1{2}\big]$. We proceed to argue that 
    \begin{equation} \label{eq:set of atoms of Fermat monoid}
        \mathcal{A}(M) = \Big\{ \frac{F_n - 2}{2^{2^{n+1}}} : n \in \nn \Big\}.
    \end{equation}
    Since $M$ is reduced, $\mathcal{A}(M)$ must be contained in the defining generating set $\Big\{ \frac{F_n - 2}{2^{2^{n+1}}}, \frac{F_n}{2^{2^{n+1}}} : n \in \nn \Big\}$ of $M$. On the other hand, for each $n \in \nn$, we see that $\frac1{2^{2^{n+1}}} \in \nn_0\big[\frac12\big] \subseteq M$ and so the element $\frac{F_n}{2^{2^{n+1}}}$ is not an atom of $M$ because
	\[
		\frac{F_n}{2^{2^{n+1}}} = 2\frac1{2^{2^{n+1}}} + \frac{F_n - 2}{2^{2^{n+1}}}.
	\]
    Thus, $\mathcal{A}(M) \subseteq \big\{ \frac{F_n - 2}{2^{2^{n+1}}} : n \in \nn \big\}$. Now fix an arbitrary $k \in \nn$, and let us verify that $\frac{F_k - 2}{2^{2^{k+1}}} \in \mathcal{A}(M)$. Since the sequence $\big(\frac{F_n - 2}{2^{2^{n+1}}}\big)_{n \ge 1}$ is strictly decreasing, $\frac{F_j - 2}{2^{2^{j+1}}} \nmid_M \frac{F_k - 2}{2^{2^{k+1}}}$ for any $j \in \ldb 1,k-1\rdb$, whence proving that $\frac{F_k - 2}{2^{2^{k+1}}}$ is an atom amounts to writing
    \begin{equation} \label{eq:k-th atom Fermat's monoid}
        \frac{F_k - 2}{2^{2^{k+1}}} = \sum_{j=k}^n b_j \frac{F_j - 2}{2^{2^{j+1}}}
    \end{equation}
    for some index $n \in \nn$ with $n \ge k$ and coefficients $b_k, \dots, b_n \in \nn_0$ with $b_k \in \{0,1\}$ and showing that $b_k=1$. Observe that we can rewrite~\eqref{eq:k-th atom Fermat's monoid} as follows
    \begin{equation} \label{eq:k-th atom Fermat's monoid I}
        - b_k \frac{F_k - 2}{2^{2^{k+1}}} + \frac1{2^{2^{k+1}}} \prod_{i=0}^{k-1} F_i = \sum_{j=k+1}^n b_j \frac{1}{2^{2^{i+1}}} \prod_{i=0}^{j-1} F_i.
    \end{equation}
    Let $p$ be an (odd) prime divisor of $F_k$, and observe that the $p$-adic valuation of the right-hand side of~\eqref{eq:k-th atom Fermat's monoid I} is positive as each summand has a factor $F_k$. This implies that $b_k \neq 0$ as otherwise the $p$-adic valuation of the left-hand side of~\eqref{eq:k-th atom Fermat's monoid I} would be zero. Hence the set of atoms of $M$ is that described in~\eqref{eq:set of atoms of Fermat monoid}, also it is clear that $\lim \frac{F_n - 2}{2^{2^{n+1}}} = 0$ while $\lim (F_n-2) = \infty$. Finally, let us verify that $M$ is not atomic by showing that we cannot write $1 \in \nn_0\big[\frac12\big] \subset M$ as a sum of finitely many atoms. Indeed, for each $n \in \nn$ the presence of the factor $F_0 = 3$ in $\frac{F_n - 2}{2^{2^{n+1}}} = 2^{-2^{n+1}}\prod_{i=0}^{n-1} F_i$ ensures that the $3$-adic valuation of every atom of $M$ is positive and so the same holds for every atomic element of $M$.
\end{ex}

Next we establish both a necessary and a sufficient condition for the atomicity of $p$-adic monoids having generating sets whose numerators are powers of the same prime. 

\begin{theorem}
	Let $p$ and $q$ be two different primes, and let $M = \langle r_n : n \in \nn \rangle$ be a $p$-adic monoid such that $\mathsf{n}(r_n)$ is a power of $q$ for every $n \in \nn$. Then the following statements hold.
	\begin{enumerate}
		\item If $M$ is atomic, then $\lim \mathsf{n}(r_n) = \infty$.
		\smallskip

		\item If $\lim \mathsf{n}(r_n) = \infty$ and $(r_n)_{n \ge 1}$ is decreasing, then $M$ is atomic.
	\end{enumerate}
\end{theorem}

\begin{proof}
	(1) Define the sequences $(\alpha_n)_{n \ge 1}$ and $(\beta_n)_{n \ge 1}$ such that $p^{\alpha_n} = \mathsf{d}(r_n)$ and $q^{\beta_n} = \mathsf{n}(r_n)$. Suppose, by way of contradiction, that $\lim \mathsf{n}(r_n) \neq \infty$. Therefore there is a natural $j$ such that $\mathsf{n(r_n)} = q^j$ for infinitely many $n \in \nn$. This implies that $\frac{q^j}{p^n} \in M$ for every $n \in \nn$. Thus, for every $x \in M^\bullet$ such that $\mathsf{n}(x) = q^m \ge q^j$, one can write
	\[
		x = \frac{q^m}{\mathsf{d}(x)} = pq^{m-j} \frac{q^j}{p \mathsf{d}(x)} \notin \mathcal{A}(M).
	\]
	As a result, every $a \in \mathcal{A}(M)$ satisfies that $\mathsf{n}(a) < q^j$. This immediately implies that $\mathcal{A}(M)$ is finite. Because $M$ is atomic with $|\mathcal{A}(M)| < \infty$, it must be finitely generated, which is a contradiction.
	\smallskip
	
	(2) Consider the subsequence $(k_n)_{n \ge 1}$ of positive integers satisfying that $\mathsf{n}(r_{k_n}) < \mathsf{n}(r_i)$ for every $i > k_n$. It follows immediately that the sequence $(\mathsf{n}(r_{k_n}))_{n \ge 1}$ is increasing. We claim that the Puiseux monoid $M := \langle r_{k_n} : n \in \nn \rangle$. Suppose that $j \notin \{k_n :n \in \nn \}$. Because $\lim \mathsf{n}(r_n) = \infty$ there are only finitely many indices $i \in \nn$ such that $\mathsf{n}(r_i) \le \mathsf{n}(r_j)$, and it is easy to see that the maximum of such indices, say $m$, belongs to $\{k_n : n \in \nn \}$. As $r_i = p^{\alpha_m - \alpha_i}q^{\beta_i - \beta_m} r_m$, it follows that $r_i \in \langle r_{k_n} : n \in \nn \rangle$. Hence $M = \langle r_{k_n} : n \in \nn \rangle$. Therefore it suffices to show that $r_{k_n} \in \mathcal{A}(M)$ for every $n \in \nn$. If
	\begin{equation} \label{eq:sufficient condition for atomicity of p-adic PM}
		\frac{q^{\beta_{k_n}}}{p^{\alpha_{k_n}}} = \sum_{i=1}^t c_i \frac{q^{\beta_{k_i}}}{p^{\alpha_{k_i}}},
	\end{equation}
	for some $t, c_1, \dots, c_t \in \nn_0$, then $t \ge n$, $c_1 = \dots = c_{n-1} = 0$, and $c_n \in \{0,1\}$. If $c_n = 0$, then by applying the $q$-adic valuation map to both sides of \eqref{eq:sufficient condition for atomicity of p-adic PM} we immediately obtain a contradiction. Thus, $c_n = 1$, which implies that $r_{k_n}$ is an atom. Hence $M$ is atomic.
\end{proof}

\medskip
\subsection{Sum of $p$-Adic: Back to the Elementary $1$-Prime Reciprocal}

Let $M$ be a Puiseux monoid. It follows that $\{\mathsf{n}(q) : q \in M\}$ is a submonoid of $M$ that is contained in $\nn_0$: we call this submonoid the \emph{numerator submonoid} of $M$. Notice that the numerator submonoid of a Puiseux monoid~$M$ is $\nn_0 \cap M$.

Let $f \colon \pp \to \nn_0$ be a function such that $p \nmid f(p)$ if $f(p) \neq 0$. Then we can consider the internal sum $\sum_{p \in \pp} \frac{f(p)}p \nn_0$ inside $\qq_{\ge 0}$ of the free commutative (Puiseux) monoids $\frac{f(p)}p \nn_0$ over the set of primes $\pp$, which is the Puiseux monoid $M_f$ defined via $f$ as follows:
\begin{equation} \label{eq:PM induced by a function f}
	M_f := \Big\langle \frac{f(p)}p : p \in \pp \Big\rangle.
\end{equation}
We call $M_f$ the Puiseux monoid \emph{induced} by the function $f$. The \emph{support} of $f$ is defined to be $\text{supp} \, f := \{p \in \pp : f(p) \neq 0\}$. One can readily check that $M_f$ is an atomic monoid with set of atoms
\begin{equation} \label{eq:set of atoms of M_f}
	\mathcal{A}(M_f) = \Big\{ \frac{f(p)}p : p \in \text{supp} f \Big\}.
\end{equation}

One important behavior of the monoid $M_f$ is that every element of $M_f$ has an elegant and helpful sum decomposition, as we show in the following proposition.

\begin{prop} \label{prop:CSD of 1-prime reciprocal PMs}
	Let $f \colon \pp \to \nn_0$ be a function with support $P$ such that $p \nmid f(p)$ for any $p \in P$, and let $M_f$ be the Puiseux monoid induced by $f$. Each $q \in M_f$ can be uniquely written as follows:
	\begin{equation} \label{eq:canonical decomposition PRPM}
		q = c_0(q) + \sum_{p \in P} c_p(q) \frac{f(p)}{p},
	\end{equation}
	where $c_0(q) \in \mathsf{n}(M)$ and all but finitely many terms of the coefficient sequence $(c_p(q))_{p \in P}$ equal zero and $c_p(q) \in \ldb 0, p-1 \rdb$ for all $p \in P$.
\end{prop}	

\begin{proof}
	To simplify notation, set $M := M_f$ and $N := \mathsf{n}(M_f)$. To argue the existence of the sum decomposition in~\eqref{eq:canonical decomposition PRPM}, first observe that in light of the atomicity of $M$, consider the set $\mathcal{S}_q$ of pairs $(c_0, (c_p)_{p \in P})$, where $c_0 \in N$ and $(c_p)_{p \in P}$ is a sequence whose terms are nonnegative integers such that
	\begin{equation} \label{eq:CSD aux existence}
		q = c_0 + \sum_{p \in P} c_p \frac{f(p)}{p}.
	\end{equation}
	Observe that $\mathcal{S}_q$ is nonempty because $M$ is atomic with $\mathcal{A}(M) = \big\{ \frac{f(p)}p : p \in P \big\}$ and $0 \in N$. It is clear that the first entry of any pair in $\mathcal{S}_q$ is at most $q$. Among all such pairs, choose $(c_0(q), (c_p(q))_{p \in P})$ maximizing the value of the first entry. Let us check that with this choice $c_p(q) < p$ for all $p \in P$: indeed, if $c_{p}(q) \ge p$ for some $p \in P$, after replacing $c_0(q)$ and $c_{p}(q))$ respectively by $c_0(q) + f(p) \in N$ and $c_{p}(q) - p \in \nn_0$, we obtain another sequence in $\mathcal{S}_q$ whose first term is strictly larger than $c_0(q)$, which is not possible. Hence, the chosen pair $(c_0(q), (c_p(q))_{p \in P})$ yields a sum decomposition as that in~\eqref{eq:canonical decomposition PRPM}.
	
	For the uniqueness of the sum decomposition, we can assume that the pair $(c'_0(q), (c'_p(q))_{p \in P})$ yields a sum decomposition of $q$ as that on the right-hand side of~\eqref{eq:canonical decomposition PRPM}. Therefore, for each $p \in P$, after clearing denominators in
	\[
		(c_0(q) - c'_0(q)) + \sum_{p \in P} (c_p(q) - c'_p(q)) \frac{f(p)}p = 0,
	\]
	we obtain that $p \mid c_p(q) - c'_p(q)$ and so $c_p(q) = c'_p(q)$. Thus, $c_0(q) = c'_0(q)$. Hence we find that every $q \in M$ can be uniquely written as in~\eqref{eq:canonical decomposition PRPM}, which concludes our proof.
\end{proof}

The sum decomposition in \eqref{eq:canonical decomposition PRPM} turns out to be quite helpful, so we introduce the following convenient terminology.

\begin{defin}
	Let $f \colon \pp \to \nn_0$ be a function with support $P$ such that $p \nmid f(p)$ for any $p \in P$, and let $M_f$ be the Puiseux monoid induced by $f$. For each element $q \in M_f$, we call the right-hand side of~\eqref{eq:canonical decomposition PRPM} the \emph{canonical sum decomposition} of $q$ in $M_f$.
	\begin{itemize}
		\item We let $c_0 : M_f \to \textsf{n}(M_f)$ be the function defined by the assignments $c_0 \colon q \mapsto c_0(q)$ for all $q \in M_f$, where $c_0(q)$ is as in~\eqref{eq:canonical decomposition PRPM}. 
		\smallskip
		
		\item For each $p \in P$, we let $c_p \colon M_f \to \nn_0$ be the function defined by the assignments $c_p \colon q \mapsto c_p(q)$ for all $q \in M_f$, where $c_p(q)$ is as in~\eqref{eq:canonical decomposition PRPM}.
	\end{itemize}
\end{defin}

The coefficients in the canonical sum decomposition of $M_f$ have the following desirable behavior with respect to divisibility.

\begin{prop} \label{prop:CSD of PRPM and divisibility}
	Let $f \colon \pp \to \nn_0$ be a function with support $P$ such that $p \nmid f(p)$ for any $p \in P$, and let $M_f$ be the Puiseux monoid induced by $f$. For $r,s \in M_f$ such that $r \mid_{M_f} s$, the following statements hold.
	\begin{enumerate}
		\item $c_0(r) \mid_{\mathsf{n}(M_f)} c_0(s)$.
		\smallskip
		
		\item If $c_p(r) > c_p(s)$ for some $p \in P$, then $c_0(r) + f(p) \le c_0(s)$.
	\end{enumerate}
\end{prop}

\begin{proof}
	For simplicity, set $M := M_f$ and $N := \mathsf{n}(M_f)$. Write $s = q+r$ for some $q \in M$. Observe that for each $p \in P$, the inequalities $0 \le c_p(q), c_p(r) < p$ allow us to pick $b_p \in \{0,1\}$ such that $0 \le c_p(q) + c_p(r) - b_p p < p$. Now write
	\begin{align*} 
		s  = q+r &= c_0(q) + c_0(r) + \sum_{p \in P} (c_p(q) + c_p(r)) \frac{f(p)}p  \\
					  &= \Big( c_0(q) + c_0(r) + \sum_{p \in P} b_p f(p) \Big) + \sum_{p \in P} \big( c_p(q) + c_p(r) - b_p \big) \frac{f(p)}p.
	\end{align*}

	(1) From the fact that $f(p) \in N$ for all $p \in P$, we deduce $c_0(q) + c_0(r) + \sum_{p \in P} b_p f(p) \in N$. In addition, $c_p(q) + c_p(r) - b_p \in \ldb 0, p-1 \rdb$ for all $p \in P$, whence the uniqueness of the canonical sum decomposition~\eqref{eq:canonical decomposition PRPM} ensures that
	\[
		c_0(s) = c_0(q) + c_0(r) + \sum_{p \in P} b_p f(p) \quad \text{ and } \quad c_p(s) = c_p(q) + c_p(r) - b_p
	\]
	for all $p \in P$. Thus, since $f(p) \in N$ for all $p \in P$, the equality $c_0(s) = c_0(q) + c_0(r) + \sum_{p \in P} b_p f(p)$ implies that $c_0(r) \mid_N c_0(s)$.
	\smallskip
	
	(2) Now assume that $c_p(r) > c_p(s)$ for some $p \in P$. Then $c_p(s) + b_p = c_p(q) + c_p(r) \ge c_p(r) > c_p(s)$, which implies that $b_p = 1$. Therefore,
    \[
		c_0(r) + f(p) = c_0(r) + b_p f(p) \le c_0(q) + c_0(r) + \sum_{p \in P} b_pf(p) = c_0(s).
	\]
\end{proof}

We proceed to show an application of the canonical sum decompositions of Puiseux monoids $M_f$ induced by functions $f \colon \pp \to \nn_0$.  We argue that such monoids satisfy the ACCP.

\begin{theorem} \label{thm:ACCP monoids}
	Let $f \colon \pp \to \nn_0$ be a function with support $P$ such that $p \nmid f(p)$ for any $p \in P$, and let~$M_f$ be the Puiseux monoid induced by $f$. Then $M_f$ satisfies the ACCP.
\end{theorem}

\begin{proof}
	For simplicity, set $M := M_f$ and $N := \mathsf{n}(M_f)$. Since $N$ is a numerical monoid, it must satisfy the ACCP.
    
    To argue that $M$ satisfies the ACCP, fix an ascending chain $(q_n + M)_{n \ge 1}$ of principal ideals of~$M$. For each $n \in \nn$, the fact that $q_{n+1} \mid_M q_n$, along with part~(1) of Proposition~\ref{prop:CSD of PRPM and divisibility} ensures that $c_0(q_{n+1}) \mid_N c_0(q_n)$. Thus, $(q_n + N)_{n \ge 1}$ is an ascending chain of principal ideals in $N$, and so it must stabilize. Hence, after dropping finitely many terms from $(q_n + M)_{n \ge 1}$, we can assume that $c_0(q_n) = c_0(q_1)$ for every $n \in \nn$. 
    
    Now consider the function $\sigma \colon M \to \nn_0$ defined as follows: $\sigma(q) := \sum_{p \in P} c_p(q)$ for all $q \in M$. It follows now from part~(2) of Proposition~\ref{prop:CSD of PRPM and divisibility} that $c_p(q_{n+1}) \le c_p(q_n)$ for every $n \in \nn$. Therefore, if the strict inclusion $q_n + M \subsetneq q_{n+1} + M$ holds for some $n \in \nn$, then the equality $c_0(q_{n+1}) = c_0(q_n)$ implies that $\sigma(q_n) > \sigma(q_{n+1})$. This, along with the fact that $\sigma(q) \ge 0$ for all $q \in M$, ensures that the set $\{n \in \nn : q_n + M \subsetneq q_{n+1} + M\}$ is finite, which is equivalent to the fact that the chain $(q_n + M)_{n \ge 1}$ stabilizes. Thus, $M$ satisfies the ACCP.
\end{proof}

It turns out that we can make choice of the function $f$ inducing Puiseux monoids $M_f$ with various factorization behavior.

\begin{ex}
    Consider the function $f \colon \pp \to \nn_0$ defined as $f(p) = 1$ for every $p \in \pp$. The Puiseux monoid induced by $f$ is the $1$-prime reciprocal $M_\pp$:
    \[
        M_f = M_\pp = \Big\langle \frac1{p} : p \in \pp \Big\rangle.
    \]
    We have seen in~\eqref{eq:set of atoms of M_f} that $\mathcal{A}(M) = \big\{ \frac1p : p \in \pp \big\}$. Also, it follows from Theorem~\ref{thm:ACCP monoids} or part~(1) of Theorem~\ref{thm:k-prime reciprocal monoids} that the Puiseux monoid $M_f$ satisfies the ACCP.  However, $M_f$ is not a BFM: for instance, for each $p \in \pp$, the element $1 \in M_f$ can be written as the sum of $p$ copies of the atom $\frac1p$, and so the inclusion $\pp \subseteq \mathsf{L}(1)$ holds (indeed, we can readily check that $\mathsf{L}(1) = \pp$ \cite[Proposition~3.1]{GG25}).
\end{ex}

Let us make a choice of $f$ such that $M_f$ is an FFM that is neither a UFM nor a finitely generated monoid.

\begin{ex}
    Now consider the function $f \colon \pp \to \nn_0$ defined as $f(p) = p-1$ for every $p \in \pp$. Then the function $f$ induces the following Puiseux monoid:
    \[
        M_f = \Big\langle 1 - \frac1p : p \in \pp \Big\rangle,
    \]
    whose set of atoms is $\big\{ \frac1p : p \in \pp \big\}$ as pointed out in~\eqref{eq:set of atoms of M_f}, and so $M_f$ cannot be finitely generated. As the set of atoms can be listed increasingly, it follows from \cite[Theorem~5.6]{fG19} that $M_f$ is an FFM. Finally, $M_f$ is not a UFM as for any distinct $p,q \in \pp$ the inclusion $\{q(p-1), p(q-1)\} \subseteq \mathsf{L}((p-1)(q-1))$ holds: in fact, the element $(p-1)(q-1)$ can be written in $M_f$ as the sum of $q(p-1)$ copies of the atom $1 - \frac1q$ or as the sum of $p(q-1)$ copies of the atom $1 - \frac1p$.
\end{ex}

We conclude this section with two examples of dense Puiseux monoids that are the internal sum of $p$-adic monoids. The first of such examples is a monoid that satisfies the ACCP but it not a BFM.

\begin{ex}
    Let $f \colon \pp \to \nn_0$ be a bounded function whose support $P$ is infinite such that $p \nmid f(p)$ for any $p \in P$. We claim that $M_f$ is a monoid that satisfies the ACCP but is not a BFM. The fact that $M_f$ satisfies the ACCP follows from Theorem~\ref{thm:ACCP monoids}. To argue that $M_f$ is not a BFM, take $n \in \nn$ such that the set $L := \{p \in P : f(p) = n \}$ is infinite (which can be done because $f$ is bounded), and observe that $n$ is an element of $M_f$ such that $L \subseteq \mathsf{L}(n)$: indeed, $n = p \frac{f(p)}p$ for all $p \in L$. Since $L$ is an infinite, so is $\mathsf{L}(n)$, which ensures that $M_f$ is not a BFM.
\end{ex}

This last example is a Puiseux monoid that has the bounded factorization property but not the finite factorization property. Although this monoid is also the internal sum of $p$-adic monoids, we were not able to find a canonical sum decomposition.

\begin{ex}
	Let $P$ be the set of odd primes, and consider the Puiseux monoid
	\[
		M := \langle A \rangle, \ \text{ where }  \ A = \bigg\{ \frac {\lfloor p/2 \rfloor}p, \frac {p-\lfloor p/2 \rfloor}p : p \in P \bigg\}.
	\]
	Clearly, $M$ is an atomic Puiseux monoid. In addition, it is not hard to verify that $\mathcal{A}(M) = A$. Observe that the element $1 \in M$ has infinitely many factorizations in $M$: indeed,
	\[
		1 = \frac {\lfloor p/2 \rfloor}p + \frac {p-\lfloor p/2 \rfloor}p 
	\]
	for every $p \in P$. Hence $M$ is not an FFM. Also, notice that $a \ge \frac13$ for every $a \in \mathcal{A}(M)$. This implies that no element $q \in M$ can be the sum of more than $\lfloor 3q \rfloor$ atoms: $\mathsf{L}(q) \subseteq \ldb 1, \lfloor 3q \rfloor \rdb$. Because $|\mathsf{L}(q)| < \lfloor 3q \rfloor$ for all $q \in M$, the Puiseux monoid $M$ is a BFM that is not an FFM.
\end{ex}

\medskip
\subsection{Sum of $p$-Adic Monoids Valuation Monoids}

Next we discuss the internal sum of $p$-adic valuation monoids. Throughout this section, for any $p \in \pp$, we let $M_p$ denote the underlying additive monoid of the rational cyclic semiring $\nn_0\big[\frac1p\big]$. Therefore, for each $p \in \pp$, the monoid $M_p$ is the Puiseux monoid consisting of all nonnegative $p$-adic rationals, and so $M_p$ is a rank-$1$ valuation monoid. In addition, for a finite nonempty set $P$ consisting of primes, we let $M_P$ denote the internal sum over $P$ of the Puiseux monoids $M_p$:
\begin{equation} \label{eq:the monoid M_P}
	M_P := \sum_{p \in P} M_p.
\end{equation}
The Puiseux monoids $M_P$ are the central objects we study in this section. It turns out that, inside these Puiseux monoids, each element has a convenient decomposition as a canonical sum, which we will describe in detail, as we proceed.

\begin{prop} \label{prop:canonical decomposition}
	Let $P$ be a finite nonempty set consisting of primes, and let $M_P$ be as in~\eqref{eq:the monoid M_P}. Then each $q \in M_P$ can be written uniquely as follows:
	\begin{equation} \label{eq:canonical representation for internal sum of p-valuation}
		q = c_0 \ + \sum_{(p,n) \in P \times \nn} c_{p,n} \frac1{p^n},
	\end{equation}
	where $c_0 \in \nn_0$ and, for each $p \in P$, the sequence $(c_{p,n})_{n \ge 1}$ consists of nonnegative integer coefficients almost all being zero such that $0 \le c_{p,n} < p$ for every $n \in \nn$.
\end{prop}

\begin{proof}
	Fix $q \in M_P$. The existence and uniqueness of the sum decomposition~\eqref{eq:canonical representation for internal sum of p-valuation} when $q=0$ is clear. Thus, we assume that $q > 0$.
	
	Because $M_P$ is the internal sum of the monoids $M_p$, we can write $q = \sum_{p \in P} m_p$, where $m_p \in M_p$ for every $p \in P$. For each $p \in P$ such that $m_p > 0$, we can write $m_p = \frac{n(p)}{p^{e(p)}}$ for some $n(p) \in \nn$ and $e(p) \in \nn_0$ such that $p \nmid n(p)$. Thus, after assuming that $\sum_{n \in \nn_0} c'_{p,n} p^n$ is the representation of $n(p)$ in base $p$, we obtain that
	\[
		m_p = \frac{n(p)}{p^{e(p)}} = \frac1{p^{e(p)}} \sum_{n \in \nn_0} c'_{p,n} p^n = \sum_{n \ge e(p)} c'_{p,n} p^{n-e(p)} + \sum_{n < e(p)} c'_{p,n} \frac{1}{p^{e(p)-n}} = c_{p,0} + \sum_{n=1}^{e(p)} c_{p,n} \frac1{p^n}, 
	\]
	where $c_{p,0} :=  \sum_{n \ge e(p)} c'_{p,n} p^{n-e(p)} \in \nn_0$ and $c_{p,k} := c'_{p, e(p)-k}$ for every $k \in \ldb 1,e(p) \rdb$. Hence the existence of \eqref{eq:canonical representation for internal sum of p-valuation} follows after taking $c_0 := \sum_{p \in P} c_{p,0}$ and $c_{p,n} = 0$ for all $(p,n) \in P \times \nn$ with $n > e(p)$.
	
	We proceed to argue the uniqueness of~\eqref{eq:canonical representation for internal sum of p-valuation}. For this, suppose that we can write some $q \in M_P$ with $q > 0$ as in~\eqref{eq:canonical representation for internal sum of p-valuation} and also as follows:
	\begin{equation} \label{eq:uniquness aux}
		q = b_0 \ + \sum_{(p,n) \in P \times \nn} b_{p,n} \frac1{p^n},
	\end{equation}
	where, for each $p \in P$, the sequence $(b_{p,n})_{n \ge 1}$ consists of nonnegative integers almost all being zero such that $b_{p,n} \in \ldb 0, p-1 \rdb$ for every $n \in \nn$. Suppose, towards a contradiction, that the sum decompositions of $q$ in~\eqref{eq:canonical representation for internal sum of p-valuation} and~\eqref{eq:uniquness aux} are not equal. In this case, we can pick $p \in P$ such that the sequences $(b_{p,n})_{n \ge 1}$ and $(c_{p,n})_{n \ge 1}$ are distinct even though
    \[
        \sum_{n \in \nn} b_{p,n} \frac1{p^n} = \sum_{n \in \nn} c_{p,n} \frac1{p^n}.
    \]
    Therefore, after taking $m := \max\{n \in \nn : b_{p,n} \neq c_{p,n} \}$, we see that
    \[
        c_{p,m} - b_{p,m} = \sum_{n=1}^{m-1} (b_{p,n} - c_{p,n})p^{m-n},
    \]
    which implies that $p \mid c_{p,m} - b_{p,m}$. However, this contradicts that $0 \le b_{p,m}, c_{p,m} < p$. Thus, the sum decomposition on the right-hand side of~\eqref{eq:canonical representation for internal sum of p-valuation} is unique.
\end{proof}

With notation as in Proposition~\ref{prop:canonical decomposition}, we will see that the sum decomposition in~\eqref{eq:canonical representation for internal sum of p-valuation} is quite helpful to understand divisibility aspects in the monoids $M_P$. Based on this, we introduce the following terminology.


\begin{defin} \label{def:CSD for the internal sum of prime-valuation PMs}
	Let $P$ be a finite nonempty set consisting of primes, and let $M_P$ be the Puiseux monoid in~\eqref{eq:the monoid M_P}. For each $q \in M_P$, we call the right-hand side of~\eqref{eq:canonical representation for internal sum of p-valuation} the \emph{canonical sum decomposition} of $q$ in $M_P$.
	\begin{itemize}
		\item We let $c_0 \colon M_P \to \nn_0$ be the function defined as follows: $c_0(q) = c_0$ for all $q \in M_P$, where $c_0$ is as in~\eqref{eq:canonical representation for internal sum of p-valuation}.
		\smallskip

        \item For each pair $(p,n) \in P \times \nn$, we let $c_{p,n} \colon M_P \to \nn_0$ be the function defined as follows: $c_{p,n}(q) = c_{p,n}$ for all $q \in M_P$, where $c_{p,n}$ is as in~\eqref{eq:canonical representation for internal sum of p-valuation}.
		
	\end{itemize}
\end{defin}

Let us take a look at some basic properties of the canonical sum decomposition we have just introduced for the Puiseux monoids $M_P$ introduced in~\eqref{eq:the monoid M_P}.

\begin{prop}
	Let $P$ be a finite nonempty set of primes, and let $M_P$ be the monoid defined in~\eqref{eq:the monoid M_P}. Then the following statements hold.
	\begin{enumerate}
		\item For each $p \in P$  and $q \in M_P$,
		\[
			v_p(q) \ge 0 \quad \text{if and only if} \quad c_{p,n}(q) = 0 \quad \text{for every} \quad n \in \nn.
		\]
		
		\item For each $(p,n) \in P \times \nn$ and for all $q,r \in M_P$ with $\gcd(\mathsf{d}(q), \mathsf{d}(r)) = 1$,
		\[
			c_{p,n}(q+r) = c_{p,n}(q) + c_{p,n}(r).
		\]
		
		\item For all $q,r \in M_P$ with $\gcd(\mathsf{d}(q), \mathsf{d}(r)) = 1$,
		\[
			c_0(q+r) = c_0(q) + c_0(r).
		\] 
	\end{enumerate}
\end{prop}

\begin{proof}
	(1) Observe that there exists $n \in \nn$ such that $c_{p,n}(q) > 0$ if and only if $v_p(q) = -m$, where $m := \max\{n \in \nn : c_{p,n}(q) > 0\}$.
	\smallskip
	
	(2) Fix a pair $(p,n) \in P \times \nn$, and take $q,r \in M_P$ with $\gcd(\mathsf{d}(q), \mathsf{d}(r)) = 1$. In light of part~(1), the fact that $p$ is a prime factor of at most one of the positive integers $\mathsf{d}(q)$ and $\mathsf{d}(r)$ implies that either $c_{p,n}(q) = 0$ or $c_{p,n}(r) = 0$. Hence $c_{p,n}(q+r) = c_{p,n}(q) + c_{p,n}(r)$.
	\smallskip
	
	(3) As before, take $q,r \in M_P$ with $\gcd(\mathsf{d}(q), \mathsf{d}(r)) = 1$. For any $(p,n) \in  \pp \times \nn$, it follows from part~(2) that $c_{p,n}(q+r) = c_{p,n}(q) + c_{p,n}(r)$. Thus,
	\[
		c_0(q+r) = q+r \ - \sum_{(p,n) \in P \times \nn} c_{p,n}(q+r)\frac1{p^n} = q+r \ - \sum_{(p,n) \in P \times \nn} c_{p,n}(q)\frac1{p^n} \ - \sum_{(p,n) \in P \times \nn} c_{p,n}(r) = c_0(q) + c_0(r).
	\]
\end{proof}

For any $q \in M_P$, given the uniqueness of the canonical sum decomposition, we can also write $q = c_0(q) + \sum_{p \in P} c_p(q)$, where $c_p(q) := \sum_{n \in \nn} c_{p,n}(q) \frac1{p^n}$. Let us now show that the canonical sum decomposition inside the monoids defined in~\eqref{eq:the monoid M_P} behaves well with respect to divisibility.

\begin{prop} \label{prop:CSD and divisibility}
	Let $P$ be a nonempty finite set of primes, and let $M_P$ be the monoid defined in~\eqref{eq:the monoid M_P}. For any $r,s \in M_P$, the following statements hold.
	\begin{enumerate}
		\item  $c_0(r)$ is the largest integer dividing $r$ in $M_P$.
		\smallskip
		
		\item If $r \mid_{M_P} s$, then $c_0(r) \le c_0(s)$.
		\smallskip
		
		\item If $r \mid_{M_P} s$ and $c_p(r) > c_p(s)$ for some $p \in P$, then $c_0(r) < c_0(s)$.
	\end{enumerate}
\end{prop}

\begin{proof}
	(1) It is clear that $c_0(r)$ is a nonnegative integer dividing $r$. In addition, if $n_r$ is a nonnegative integer such that $n_r \mid_{M_P} r$, then we can write $r = n_r + r'$ for some $r' \in M_P$ and observe that
    \[
        r = (n_r + c_0(r')) + \ \sum_{(p,n) \in P \times \nn} c_{p,n}(r')\frac1{p^n}.
    \]
    Thus, from the uniqueness of the canonical sum decomposition of $r$, we obtain that $c_0(r) = n_r + c_0(r')$, and so $n_r \le c_0(r)$. Therefore $c_0(r)$ is the largest integer dividing $r$ in $M_P$.
	\smallskip
	
	(2) Write $s = q+r$ for some $q \in M_P$. Then we can write
	\[
		s = c_0(q) + c_0(r) + \ \sum_{(p,n) \in P \times \nn} (c_{p,n}(q) + c_{p,n}(r)) \frac1{p^n}.
	\]
	As a consequence, $c_0(q) + c_0(r)$ divides $s$ in $M_S$. Therefore, part~(1) guarantees that the inequality $c_0(s) \ge c_0(q) + c_0(r) \ge c_0(r)$ holds.
	\smallskip
	
	(3) Now suppose that $r \mid_{M_P} s$ and $c_p(r) > c_p(s)$ for some $p \in P$. Assume, by way of contradiction, that $c_0(r) \ge c_0(s)$ and so that $c_0(r) = c_0(s)$. Set $r' := s-r \in M_P$ and observe that $c_0(r') = 0$ and we can write $c_p(r) + c_p(r') = n_p + r_p$ for some $n_0 \in \nn_0$ and $r_p \in \nn_0\big[\frac1p\big] \cap (0,1)$, and so the fact that $c_0(r) = c_0(s) \ge n_p + c_0(r)$ ensures that $n_p = 0$. Thus, $c_p(r) + c_p(r') < 1$ and this implies that $c_p(s) = c_p(r) + c_p(r') \ge c_p(r)$, a contradiction.
\end{proof}

We conclude this section proving that for each finite nonempty set $P$ consisting of primes, $M_P$ is an antimatter monoid and also that $M_P$ is a valuation monoid if and only if $|P| = 1$.

\begin{prop} \label{prop:M_P is antimatter and M_P is a ValMon if |P|=1}
	Let $P$ be a finite nonempty set of primes, and let $M_P$ be the monoid defined in~\eqref{eq:the monoid M_P}. Then the following statements hold.
	\begin{enumerate}
		\item $M_P$ is an antimatter.
		\smallskip
		
		\item $M_P$ is a valuation monoid if and only if $|P| = 1$.
	\end{enumerate}
\end{prop}

\begin{proof}
	 (1) As $M_P$ is the internal sum over $P$ of the Puiseux monoids $M_p$, the fact that $M_P$ is antimatter follows immediately from the fact that, for each $p \in P$, the monoid $M_p$ is antimatter: indeed, none of elements in the generating set $\{\frac1{p^n} : n \in \nn_0\}$ of $M_p$ is an atom of $M_p$ because $\frac1{p^n} = p\frac1{p^{n+1}}$ for every $n \in \nn_0$.
	 \smallskip
	 
	 (2) For the reverse implication, it suffices to observe that if $P$ equals a singleton $\{p\}$ for some $p \in \pp$, then $M_P$ is the additive monoid $M_p$ of $\nn_0[\frac1p]$, which is clearly a valuation monoid. For the direct implication, assume that $|P| \ge 2$. Let the pair $(M_P, \mid_{M_P})$ be the partially ordered with underlying set $M_P$ and order relation given by divisibility inside $M_P$. Let us argue the following claim, from which we can immediately complete our proof.
	 \smallskip
	 
	 \noindent \textsc{Claim.} For any distinct $p_1, p_2 \in P$ and $(q_1,q_2) \in M_{p_1} \times M_{p_2}$ with $c_0(q_1) = c_0(q_2)$, the following conditions are equivalent.
	 \begin{enumerate}
	 	\item[(a)] $\{q_1, q_2\}$ is a chain in the poset $(M_P, \mid_{M_P})$.
	 	\smallskip
	 	
	 	\item[(b)] $\{q_1, q_2\}$ intersects $\nn_0$.
	 \end{enumerate}
	 \smallskip
	 
	 \noindent \textsc{Proof of Claim.} (a) $\Rightarrow$ (b): Suppose that $\{q_1, q_2\}$ is a chain in the poset $(M_P, \mid_{M_P})$ and assume, without loss of generality, that $q_1 \mid_{M_P} q_2$. Given the uniqueness of the canonical sum decompositions of $q_1$ and $q_2$ inside $M_P$, the equality $c_{p,n}(q_1) = 0$ holds for all $(p,n) \in P \times \nn$ with $p \neq p_1$ while the equality $c_{p,n}(q_2) = 0$ for all $(p,n) \in P \times \nn$ with $p \neq p_2$. Hence
	 \[
	 	q_1 = c_0(q_1) + c_{p_1}(q_1)  \quad \text{ and } \quad q_2 = c_0(q_2) +  c_{p_2}(q_2), 
	 \]
	 where $c_{p_1}(q_1) = \sum_{n \in \nn} c_{p_1,n}(q_1) \frac1{p_1^n} \in M_{p_1}$ and $c_{p_2}(q_2)  = \sum_{n \in \nn} c_{p_2,n}(q_2) \frac1{p_2^n} \in M_{p_2}$. Since $q_1 \mid_{M_P} q_2$ and $c_0(q_1) = c_0(q_2)$, it follows from part~(3) of Proposition~\ref{prop:CSD and divisibility} that $c_{p_1}(q_1) \le c_{p_1}(q_2) = 0$, which enforces the equality $c_{p_1}(q_1) = 0$. Thus, $q_1 = c_0(q_1) \in \nn_0$ and so $q_1 \in \{q_1, q_2\} \cap \nn_0$.
	 
	 (b) $\Rightarrow$ (a): Now suppose that $\{q_1, q_2\}$ intersects $\nn_0$ and assume, without loss of generality, that $q_1 \in \nn_0$. Then $c_0(q_2) = c_0(q_1) = q_1$ by part~(1) of Proposition~\ref{prop:CSD and divisibility}. Thus, $q_1 = c_0(q_2) \mid_{M_P} q_2$, which implies that $\{q_1,q_2\}$ is a chain in the poset $(M_P, \mid_{M_P})$. Hence the claim is established.
	 \smallskip
	 
	 Finally, we argue that $M_P$ is not a valuation monoid. Take $p_1, p_2 \in P$ with $p_1 \neq p_2$ and then set $q_1 := \frac1{p_1}$ and $q_2 := \frac1{p_2}$. Note that $(q_1,q_2) \in M_{p_1} \times M_{p_2}$ and $c_0(q_1) = c_0(q_2) = 0$. Thus, by virtue of the established claim, the fact that neither $q_1$ nor $q_2$ are integers implies that $q_1 \nmid_{M_P} q_2$ and $q_2 \nmid_{M_P} q_1$. Hence $M$ is not a valuation monoid.
\end{proof}

\bigskip
\section{Multiplicatively Closed Puiseux Monoids}
\label{sec:multiplicatively cyclic PM}

The first study of the atomic structure of a family of multiplicatively closed Puiseux monoids was briefly initiated by the second and third authors in~\cite[Section~5]{GG18} and continued by the first three authors in~\cite{CGG20} and by fourth author in~\cite{hP20}, and all these paper focused on the additive monoids of the rational cyclic semiring $\nn_0[q]$. Our purpose in this section is to revisit this class and explore other new classes of Puiseux monoids that are multiplicatively closed.

\medskip
\subsection{Additive Monoids of Rational Cyclic Semirings}

Let us start by considering, for each $q \in \qq_{> 0}$, the cyclic subsemiring of $\qq_{\ge 0}$ generated by $q$:
\[
	\nn_0[q] := \{f(q) : f(x) \in \nn_0[x]\},
\]
where $\nn_0[x]$ denotes the semiring of polynomials in the indeterminate $x$ over $\nn_0$. We let $M_q$ denote the underlying additive monoid of the semiring $\nn_0[q]$, which is a Puiseux monoid. Observe that the monoid $M_q$ is additively generated by the powers of $q$:
\begin{equation} \label{eq:additive monoid of N_0[q]}
    M_q = \big\langle q^n : n \in \nn_0 \big\rangle.
\end{equation}
One can readily see that if $\mathsf{n}(q) = 1$, then $M_q = \nn_0\big[\frac1d\big]$, where $d := \mathsf{d}(q)$, and so any nonzero $r \in M_q$ can be written as $r = d\big(\frac{r}{d})$ and so $M_q$ is an antimatter monoid. In addition, for any $r,s \in M_q$, the divisibility relation $r \mid_{M_q} s$ holds if and only if $r \le s$, and so $M_q$ is a valuation monoid.
\smallskip

As was the case for various Puiseux monoids we investigated in previous sections, elements inside the monoid $M_q$ have a special sum decomposition (when $q \notin \nn$), which we argue as follows.

\begin{prop} \label{prop:CSD of multiplicative cyclic}
    For $q \in \qq_{>0} \setminus \nn$, let $M_q$ be the underlying additive monoid of the semiring $\nn_0[q]$. Then each $r \in \nn_0[q]$ can be uniquely written as follows:
	\begin{equation} \label{eq:CSD in N[q]}
		r = c_0(r) + \sum_{n \in \nn} c_n(r) q^n,
	\end{equation}
	where $(c_n(r))_{n \ge 0}$ is a sequence with $c_n(r) \in \ldb 0, \mathsf{d}(q)-1 \rdb$ for every $n \in \nn$ such that all but finitely many of whose terms are zero.
\end{prop}

\begin{proof}
    To argue the existential part of the statement, it is convenient to split our argument into two cases.
    \smallskip
    
    \textsc{Case 1:} $\mathsf{n}(q) = 1$. After setting $d := \mathsf{d}(q)$, we see that $M_q = \big\langle \frac1{d^n} : n \in \nn_0 \big\rangle$. Fix $r \in M_q$ and take the minimum $k \in \nn_0$ such that $d^kr \in \nn_0$. It follows from the minimality of $k$ that we can represent $d^kr$ in base $d$ as $d^kr = \sum_{n=0}^k c_n d^{k-n}$ for some coefficients $c_0, \dots, c_k \in \ldb 0,d-1 \rdb$. From this, we deduce that
    \[
        r = \frac1{d^k} \sum_{n=0}^k c_n d^{k-n} = c_0 + \sum_{n=1}^k c_n \frac1{d^n} = c_0 + \sum_{n \in \nn} c_n q^n,
    \]
    where the equality $c_n = 0$ holds for every $n > k$. After setting $c_0(r) := c_0$ and $c_n(r) := c_n$, one obtains the sum decomposition of~\eqref{eq:CSD in N[q]}.
    \smallskip
    
    \textsc{Case 2:} $\mathsf{n}(q) \ge 2$. As $\mathsf{n}(q), \mathsf{d}(q) \ge 2$, it follows from~\cite[Theorem~6.2]{GG18} that the Puiseux monoid~$M_q$ is atomic with set of atoms
    \[
        \mathcal{A}(M_q) = \big\{q^n : n \in \nn_0 \big\}.
    \]
    Fix an element $r \in M_q$, and let us find a sum decomposition of $r$ as that in~\eqref{eq:CSD in N[q]}. Since $M_q$ is atomic, $\mathsf{Z}(r)$ is a nonempty set. For each factorization $z := \sum_{n \in \nn_0} c_n q^n \in \mathsf{Z}(r)$, consider the subset
    \[
        N_z := \{0\} \cup \{n \in \nn : c_n \ge \mathsf{d}(q) \}
    \]
    of $\nn_0$. Take a factorization $z := \sum_{n \in \nn_0} c_n q^n \in \mathsf{Z}(r)$ such that the minimum $m := \min N_z$ is as small as it can possibly be. We claim that $m = 0$. Assume, towards a contradiction, that $m \ge 1$. Since $c_m \ge \mathsf{d}(q)$, we can write $c_m = b_m \mathsf{d}(q) + r_m$ for some $b_m \in \nn$ and $r_m \in \ldb 0, \mathsf{d}(q) - 1 \rdb$. Observe that
    \[
        c_{m-1}q^{m-1} + c_m q^m = c_{m-1}q^{m-1} + (b_m \mathsf{d}(q))q^m + r_m q^m = (c_{m-1} + b_m \mathsf{n}(q))q^{m-1} + r_m q^m.
    \]
    Because $r_m < \mathsf{d}(q)$, after replacing $c_{m-1}q^{m-1} + c_mq^m$ by $(c_{m-1} + b_m \mathsf{n}(q))q^{m-1} + r_m q^m$ in $z$, we obtain a new factorization $w$ such that $\min N_w < m$, which contradicts the minimality of $m$. Hence $m=0$, which implies that the coefficients in the right-hand side of
    \begin{equation} \label{eq:CSD in N[q] I}
        r = c_0 + \sum_{n \in \nn} c_n q^n
    \end{equation}
    are such that $c_n < \mathsf{d}(q)$ for every $n \in \nn$, and so after setting $c_0(r) := c_0$ and $c_n(r) := c_n$, we obtain the sum decomposition of~\eqref{eq:CSD in N[q]}.
    
    Let us now argue the uniqueness of the sum decomposition in~\eqref{eq:CSD in N[q]}. Assume, by way of contradiction, that the element $r \in M_q$ has two distinct sum decompositions: the one shown in~\eqref{eq:CSD in N[q]} and also the following
    \begin{equation} \label{eq:CSD II for uniqueness}
        r = b_0(r) + \sum_{n \in \nn} b_n(r) q^n,
    \end{equation}
    where $(b_n(r))_{n \ge 0}$ is a sequence with $b_n(r) \in \ldb 0, \mathsf{d}(q)-1 \rdb$ for every $n \in \nn$ such that all but finitely many terms of $(b_n(r))_{n \ge 0}$ are zero. After setting $m := \max\{n \in \nn_0 : |c_n - b_n| \neq 0 \}$ and set equal the right-hand sides of~\eqref{eq:CSD in N[q]} and~\eqref{eq:CSD II for uniqueness}, we obtain
    \[
        (c_m(r) - b_m(r))\mathsf{n}(q)^m = \sum_{n=0}^{m-1} (b_n(r) - c_n(r))\mathsf{n}(q)^n\mathsf{d}(q)^{m-n}.
    \]
    As $\mathsf{d}(q)$ divides every summand in the right-hand side of the obtained identity, $\mathsf{d}(q)$ must divide the right-hand side of the same identity, whence $\mathsf{d}(q) \mid c_m - b_m$ because $\gcd(\mathsf{d}(q), \mathsf{n}(q)) = 1$. However, this implies that $b_m = c_m$, which is a contradiction.
\end{proof}

We can now use the canonical sum decomposition established in~\eqref{eq:CSD in N[q]} for the monoids $M_q$ to argue that it is atomic if and only if it is an LFFM.

\begin{theorem} \label{thm:multiplicatively cyclic are LFF iff are atomic}
    For $q \in \qq_{> 0}$, let $M_q$ be the underlying additive monoid of $\nn_0[q]$. Then the following conditions are equivalent.
    \begin{enumerate}
        \item[(a)] $M_q$ is an LFFM.
        \smallskip

        \item[(b)] $M_q$ is atomic.
        \smallskip

        \item[(c)] $q \in \qq \setminus \{\frac1n : n \in \nn_{\ge 2} \}$.
    \end{enumerate}
\end{theorem}

\begin{proof}
    (a) $\Rightarrow$ (b): It follows from the corresponding definitions.
    \smallskip

    (b) $\Rightarrow$ (c): If $q = \frac1{n}$ for some $n \in \nn$ such that $n \ge 2$, then $\frac1{n^k} = n \frac1{n^{k+1}}$ for every $k \in \nn$, and so the set of atoms of $M_q$ is empty. Hence $M_q$ is not atomic.
    \smallskip

    (c) $\Rightarrow$ (a): Finally, assume that $q \in \qq \setminus \big\{ \frac1n : n \in \nn_{\ge 2} \big\}$. In this case, it is well known and not hard to verify that the monoid $M_q$ is atomic. We split the rest of our argument into the following two cases.
    \smallskip
    
    \textsc{Case 1:} $q \ge 1$. In this case, we can argue that $M_q$ is indeed an FFM, whence an LFFM. Observe that if $q \in \nn$, then $M_q = \nn_0$ and so it is a UFM and, in particular, an FFM. Therefore we can assume that $q \in \qq_{> 1} \setminus \zz$. In this case, the set of atoms of $M_q$ is
    \begin{equation} \label{eq:atoms of multiplicatively cyclic}
        \mathcal{A}(M_q) = \big\{ q^n : n \in \nn_0 \big\}.
    \end{equation}
    Therefore $M_q$ is an increasing positive monoid of an ordered field because the sequence of atoms $(q_n)_{n \ge 0}$ is increasing. Hence $M_q$ is an FFM by virtue of \cite[Theorem~5.6]{fG19}.
    \smallskip
    
    \textsc{Case 2:} $q < 1$. Assume, towards a contradiction, that $M_q$ is not an LFFM in this case. Using the fact that $q \neq \frac1n$ for any $n \in \nn$ with $n \ge 2$, one can verify that the set of atoms of $M_q$ is given by the equality~\eqref{eq:atoms of multiplicatively cyclic} and, therefore, $M_q$ is atomic.
    
    Since $M_q$ is not an LFFM, we can choose $\ell \in \nn$ such that $|\mathsf{Z}_\ell(r)| = \infty$ for some $r \in M_q$. We can assume that we have chosen $\ell$ as small as possible. For each index $m,n \in \nn_0$, we let $\mathsf{Z}_{\ell,m}(r, q^n)$ denote the set of factorizations in $\mathsf{Z}_\ell(r)$ having exactly $m$ copies of the atom~$q^n$. Observe that for each $n \in \nn_0$, the set $\mathsf{Z}_{\ell,m}(r,q^n)$ is empty for every $m \ge \ell$. Therefore, for each index $n \in \nn_0$, the equality
    \[
        \mathsf{Z}_\ell(r) = \bigcup_{m=0}^\ell \mathsf{Z}_{\ell,m}(r,q^n)
    \]
    holds. Fix $n \in \nn_0$. Since $|\mathsf{Z}_\ell(r)| = \infty$, there is an index $m \in \ldb 0,\ell \rdb$ such that $|\mathsf{Z}_{\ell, m}(r,q^n)| = \infty$. This implies, in particular, that $r - mq^n$ belongs to~$M_q$. Define the map $f \colon \mathsf{Z}_{\ell,m}(r,q^n) \to \mathsf{Z}_{\ell-m}(r-mq^n)$ as follows:
    \[
        f \colon m q^n + \sum_{k \in \nn_0} c_k q^k \mapsto \sum_{k \in \nn_0} c_k q^k,
    \]
    for every sequence $(c_k)_{k \ge 0}$ of nonnegative integers such that $c_n = 0$ and $mq^n + \sum_{k \in \nn_0} c_k q^k$ is a length-$\ell$ factorization of~$r$. It is clear that $f$ is an injective function, and so $|\mathsf{Z}_{\ell,m}(r,q^n)| \le |\mathsf{Z}_{\ell-m}(r-mq^n)|$. Thus, $\mathsf{Z}_{\ell-m}(r-mq^n)$ is an infinite set consisting of length-$(\ell-m)$ factorizations of $r-mq^n$. Therefore it follows from the minimality of $\ell$ that $m=0$. Thus, for each $n \in \nn_0$,
    \begin{equation} \label{eq:aux size of fact sets}
        |\mathsf{Z}_{\ell,0}(r,q^n)| = \infty \quad \text{and} \quad \bigg{|} \bigcup_{m \in \nn} \mathsf{Z}_{\ell,m}(r,q^n) \bigg{|} = \bigg{|} \bigcup_{m=1}^\ell \mathsf{Z}_{\ell,m}(r,q^n) \bigg{|} < \infty
    \end{equation}
    As $q < 1$, we can choose an index $N \in \nn$ large enough so that $q^N < r/\ell$. Now let $\mathsf{Z}_{\ell,\ge N}(r)$ be the set consisting of factorizations in $\mathsf{Z}_\ell(r)$ which do not contain any of the atoms $1, q, \dots, q^N$. It follows from~\eqref{eq:aux size of fact sets} that $\mathsf{Z}_{\ell,m}(r,q^n)$ is a finite set for all $(m,n) \in \nn \times \ldb 0, N-1 \rdb$. This implies that
    \[
        \mathsf{Z}_{\ell,\ge N}(r) = \mathsf{Z}_{\ell}(r) \setminus \bigg( \bigcup_{n=0}^{N-1} \bigcup_{m \in \nn} \mathsf{Z}_{\ell,m}(r,q^n) \bigg),
    \]
    whence the set $\mathsf{Z}_{\ell,\ge N}(r)$ is nonempty. Take a sequence $(c_k)_{k \ge N}$ of nonnegative terms such that $\sum_{k \ge N} c_k q^k$ is a length-$\ell$ factorization of $r$ in $M_q$, and notice that
    \[
        r = \sum_{k \ge N} c_k q^k \le \sum_{k \ge N} c_k q^N = q^N \ell,
    \]
    which contradicts our choice of $N \in \nn$ such that $q^N < r/\ell$. Hence we conclude that $M_q$ is an LFFM.
\end{proof}

\medskip
\subsection{Internal Sum of Rational Cyclic Semirings}                                                           

Various classes of positive monoids that generalize the class consisting of the additive monoids of rational cyclic semirings have been studied in recent literature (see, for instance, \cite{ABLST23,ABP21,hP23,hP20}). In this final subsection, we study another natural generalization of such class, the Puiseux monoids generated by the nonnegative powers of finitely many positive rationals.

Fix $n \in \nn$ and let $x_1, \dots, x_n$ be $n$ distinct indeterminates or, more formally, let $X := \{x_1, \dots, x_n \}$ with $|X| = n$ be a subset of a field extension of $\qq$ such that $X$ is algebraically independent over $\qq$. As usual, we let $\nn_0[x_1, \dots, x_n]$ denote the polynomial semiring in the indeterminates $x_1, \dots, x_n$ over the prototypical semiring $\nn_0$. For positive rationals $q_1, \dots, q_n \in \qq_{> 0}$, set $Q := \{q_1, \dots, q_n\}$. Then
\[
	\nn_0[Q] := \nn_0[q_1, \dots, q_n] := \big\{ f(q_1, \dots, q_n) : f \in \nn_0[x_1, \dots, x_n] \big\}
\]
is the semiring extension of $\nn_0$ by $Q$ inside~$\qq$, and so the underlying additive monoid of $\nn_0[Q]$ is a Puiseux monoid. The monoids we proceed to introduce and study are submonoids of the underlying additive monoid of $\nn_0[Q]$.

\begin{defin}
	For a nonempty finite set $Q$ consisting of positive rationals, let $M_Q$ be the Puiseux monoid generated by the nonnegative powers of the elements of $Q$:
	\begin{equation} \label{eq:Q-power monoid}
		M_Q = \big\langle q^n : (q,n) \in Q \times \nn_0 \big\rangle.
	\end{equation}
	We call $M_Q$ the (Puiseux) monoid \emph{powerly generated} by $Q$ or a \emph{powerly generated} Puiseux monoid.
\end{defin}

Throughout the rest of the paper, we will use the following notation: for any $q \in \qq_{> 0}$, we write~$M_q$ instead of $M_{\{q\}}$. Observe that this is consistent with the notation used in the previous section, where we let $M_q$ denote the underlying additive monoid of the rational cyclic semiring $\nn_0[q]$. Thus, for each finite nonempty set $Q$ consisting of positive rationals, the monoid powerly generated by $Q$ is simply the internal sum over $Q$ of the Puiseux monoids $M_q$:
\[
	M_Q = \sum_{q \in Q} \big\langle q^n : n \in \nn_0 \big\rangle = \sum_{q \in Q} M_q.
\]
We have already considered a special class of powerly generated monoids: indeed, for any finite nonempty set~$P$ consisting of primes, we can set $Q := \big\{\frac1p : p \in P \big\}$ and notice that the monoid $M_Q$ powerly generated by $Q$ is the internal sum over $p \in P$ of the valuation Puiseux monoids $\nn_0\big[ \frac1p \big]$, which we already considered at the end of Section~\ref{sec:p-adic PM}. Although we have seen in Proposition~\ref{prop:M_P is antimatter and M_P is a ValMon if |P|=1} that these special powerly generated monoids are antimatter, this is not the case for more general powerly generated monoids. Let us take a look at a powerly generated monoid that is atomic but does not satisfy the ACCP.

\begin{ex}
	Fix two distinct positive integers $n_1, n_2 \in \nn$ such that $\gcd(n_1, n_2) = 1$, and take $d_1,d_2 \in \nn$ such that $n_1 n_2 < d_2$. Then consider the Puiseux monoid
	\[
		M := \bigg\langle \bigg(\frac{n_1}{d_1 d_2}\bigg)^m, \bigg(\frac{n_2}{d_1 d_2}\bigg)^m : m \in \nn_0 \bigg\rangle.
	\] 
	Fix $k \in \nn$, and consider the numerical monoid $N := \big\langle n_1^k, n_2^k \big\rangle$. Now observe that we can bound the Frobenius number $\mathsf{f}(N)$ as follows: $\mathsf{f}(N) < (n_1^k - 1)(n_2^k - 1) < d_2^k$. Therefore there exist coefficients $c_1, c_2 \in \nn_0$ such that $c_1 n_1^k + c_2 n_2^k = d_2^k$. This implies that
	\[
		\frac 1{d_1^k} = \frac{c_1 n_1^k + c_2 n_2^k}{(d_1 d_2)^k} = c_1 \bigg(\frac{n_1}{d_1 d_2}\bigg)^k + c_2 \bigg(\frac{n_2}{d_1 d_2}\bigg)^k \in M.
	\]
	We have verified that $\frac{1}{d_1^m} \in M$ for every $m \in \nn$, from which we obtain that $\big( \frac1{d_1^m} + M \big)_{m \ge 1}$ is an ascending chain of principal ideals of $M$ that does not stabilize. As a consequence, $M$ is an atomic monoid that does not satisfy the ACCP.
\end{ex}

Unlike additive monoids of rational cyclic semirings, for each nonempty finite subset $Q$ of $\qq_{> 0}$ with $|Q| \ge 2$, the monoid $M_Q$ powerly generated by $Q$ is not closed under the standard multiplication. We proceed to characterize when $M_Q$ is closed under the standard multiplication.

\begin{prop}
	For a nonempty finite set $Q$ consisting of positive rationals, the following conditions are equivalent.
	\begin{enumerate}
		\item[(a)] $M_Q$ is closed under multiplication, which means that $M_Q = \nn_0[Q]$.
		\smallskip
		
		\item[(b)] $q^i r^j \in M_Q$ for all $q,r \in Q$ and $i,j \in \nn_0$.
	\end{enumerate}
\end{prop}

\begin{proof}
	(a) $\Rightarrow$ (b): This is straightforward.
	\smallskip
	
	(b) $\Rightarrow$ (a): Assume now that $q^i r^j \in M_Q$ for all $q,r \in Q$ and $i,j \in \nn_0$. If $Q = \{q\}$ for some $q \in \qq_{> 0}$, then $M_Q = \nn_0[q]$, which is a semiring. Therefore we assume that $|Q| \ge 2$. 
	Let us show that, for any $k \in \nn_{\ge 2}$ and $q_1,\dots, q_k \in Q$, the product $q_1 \cdots q_k$ belongs to $M_Q$: we proceed by inducting on $k$. The case $k=2$, which is our base case, follows from our initial assumption. For the inductive step, fix $k \in \nn$ with $k \ge 2$ and suppose that, for any $\ell \in \ldb 2,k \rdb$, the product of any~$\ell$ elements of $Q$ (repetitions allowed) belongs to $M_Q$. Take $q_1, \dots, q_{k+1} \in Q$ and let us verify that $q := q_1 \cdots q_{k+1}$ belongs to $M_Q$. If $q_1 = \dots = q_{k+1}$, then $q = q_1^{k+1} \in M_Q$. Otherwise, after relabeling the subindices of $q_1, \dots, q_{k+1}$, we can take $e \in \nn$ with $e < k+1$ such that $q = (q_1 \cdots q_{k+1-e})q_{k+1}^e$ and $q_{k+1} \notin \{q_1, \dots, q_{k+1-e} \}$. As $1 \le e < k+1$, our induction hypothesis ensures that $q_1 \cdots q_{k+1-e} \in M_Q$, Therefore we can take a finite-supported sequence $(c_{q,n})_{n \ge 0}$ of nonnegative integer coefficients such that
	\[
		q = q_{k+1}^e \sum_{(q,n) \in Q \times \nn_0} c_{q,n} q^n = \sum_{(q,n) \in Q \times \nn_0} c_{q,n} q^n q_{k+1}^e,
	\]
	and so we can deduce that $q \in M_Q$ from the fact that $q^n q_{k+1}^e \in M_Q$ for all $(q,n) \in Q \times \nn_0$. Hence $c \prod_{q \in Q} q^{e_q} \in M_Q$ for all coefficient $c \in \nn_0$ and $Q$-tuple $(e_q)_{q \in Q}$ with nonnegative integer entries, and so $M_Q$ is closed under the standard multiplication.
	
	Therefore $M_Q$ is a subsemiring of~$\qq$ containing both $\nn_0$ and $Q$, and so the inclusion $\nn_0[Q] \subseteq M_Q$ follows from the fact that $\nn_0[Q]$ is the smallest subsemiring of $\qq$ satisfying the same condition. The other inclusion follows immediately, so $M_Q = \nn_0[Q]$.
\end{proof}

Observe that if $k \in Q$ for some $k \in \nn_0$, then $k^n \in \nn_0$ for all $n \in \nn_0$. Thus, $M_Q = \nn_0$ when $Q \subset \nn$
and $M_Q = M_{Q \setminus \nn}$ if $Q \not\subseteq \nn$. Therefore we can restrict our attention, without loss of generality, to monoids powerly generated by finite nonempty subsets of $\qq_{> 0} \setminus \nn$. Using the fact that, for any nonempty finite subset $Q$ of $\qq_{> 0}$, the monoid powerly generated by $Q$ is the internal sum of $|Q|$ additive monoids of rational cyclic semirings, we can simultaneously extend the canonical sum decompositions established in Propositions~\ref{prop:canonical decomposition} and~\ref{prop:CSD of multiplicative cyclic}.

\begin{prop} \label{prop:CSD for Q-power monoids}
	Let $Q$ be a subset of $\qq_{> 0} \setminus \nn$ such that $2 \le |Q| < \infty$ and $\gcd(\mathsf{d}(Q)) = 1$. Then each $r \in M_Q$ can be uniquely written as follows:
	\begin{equation} \label{eq:CSD for $Q$-power monoids}
		r = c_0(r) \ + \sum_{(q,n) \in Q \times \nn} c_{q,n}(r) q^n,
	\end{equation}
	where $c_0(r) \in \nn_0$ and, for each $q \in Q$, the sequence $(c_{q,n}(r))_{n \ge 1}$ is finite-supported with nonnegative integer terms such that $c_{q,n} \in \ldb 0, \mathsf{d}(q) - 1 \rdb$ for every $q \in Q$.
\end{prop}

\begin{proof}
	Fix a nonzero $r \in M_Q$, and let us argue that $r$ has a unique sum decomposition as the one specified in~\eqref{eq:CSD for $Q$-power monoids}. When $Q$ consists of only one element, namely $q$, then the desired sum decomposition is the canonical sum decomposition of $r$ inside $M_q$, which we have already established in Proposition~\ref{prop:CSD of multiplicative cyclic}. Thus, for the rest of the proof we assume that $|Q| \ge 2$.
	\smallskip
	
	For the existence, first take a $Q$-tuple $(r_q)_{q \in Q}$ with $r_q \in M_q$ for all $q \in Q$ such that $r = \sum_{q \in Q} r_q$, which is possible because $M_Q = \sum_{q \in Q} M_q$. Then, for each $q \in Q$, take $c_0(r_q) \in \nn_0$ and $(c_{q,n}(r_q))_{n \ge 1}$ such that the right-hand side of equality
	\[
		c_0(r_q) + \sum_{n \in \nn} c_{q,n}(r_q) q^n
	\]
	is the canonical sum decomposition of $r_q$ in $M_q$ given in~\eqref{eq:CSD in N[q]}. Then
	\begin{equation} \label{eq:existence part CSD Q-power}
		r = \sum_{q \in Q} r_q = \sum_{q \in Q} \Big( c_0(r_q) + \sum_{n \in \nn} c_{q,n}(r_q) q^n \Big) = \sum_{q \in Q} c_0(r_q) \ + \sum_{(q,n) \in Q \times \nn} c_{q,n}(r_q) q^n.
	\end{equation}
	It is clear that $\sum_{q \in Q} c_0(r_q) \in \nn_0$. In addition, for each $q \in Q$, the fact that $c_0(r_q) + \sum_{n \in \nn} c_{q,n}(r_q) q^n$ is the canonical sum decomposition of $r_q$ in $M_q$ guarantees that $c_{q,n}(r_q) \in \ldb 0, \mathsf{d}(q)-1 \rdb$. Thus, we can obtain the desired sum decomposition of $r$ in $M_Q$ from the rightmost part of~\eqref{eq:existence part CSD Q-power} after setting $c_0(r) := \sum_{q \in Q} c_0(r_q)$ and $c_{q,n}(r) := c_{q,n}(r_q)$ for all $(q,n) \in Q \times \nn$.
	\smallskip
	
	To prove the uniqueness of the sum decomposition given in~\eqref{eq:CSD for $Q$-power monoids}, suppose we can write element $r$ inside $M_Q$ in the following two ways:
	\begin{equation} \label{eq:existence of CSD of Q-power monoids}
		r = b_0(r) \ + \sum_{(q,n) \in Q \times \nn} b_{q,n}(r) q^n \quad \text{ and } \quad r = c_0(r) \ + \sum_{(q,n) \in Q \times \nn} c_{q,n}(r) q^n,
	\end{equation}
	where $b_0(r), c_0(r) \in \nn_0$ and, for each $q \in Q$, the sequences $(b_{q,n}(r))_{n \ge 1}$ are finitely supported such that $b_{q,n}(r), c_{q,n}(r) \in \ldb 0, \mathsf{d}(q)-1 \rdb$ for every $n \in \nn$. For each $q \in Q$, the sum $\sum_{n \in \nn}(b_{q,n}(r) - c_{q,n}(r))q^n$ belongs to $\nn_0[q]$, and so it can be written as $a_q \frac{1}{\mathsf{d}(q)^{e_q}}$ for some $a_q,e_q \in \nn_0$. We can further assume that the pair $(a_q, e_q)$ has been chosen in $\nn_0^2$ so that the exponent~$e_q$ is as small as it can possibly be. Note that, for each $q \in Q$ such that $a_q \neq 0$, the minimality of $e_q$ implies that $\mathsf{d}(q) \nmid a_q$. Set
	\[
		m := \prod_{q \in Q} \mathsf{d}(q)^{e_q} \quad \text{and} \quad n_q := \frac{m}{\mathsf{d}(q)^{e_q}}
	\]
	for every $q \in Q$. As $\gcd(\mathsf{d}(Q)) = 1$, for each $q \in Q$, we observe that $n_q \in \nn$ with $\gcd(\mathsf{d}(q)^{e_q},n_q) = 1$ and also that, for each $t \in Q$, the divisibility relation $\mathsf{d}(t)^{e_t} \mid n_q$ holds if and only if $t \neq q$. After subtracting the sum decompositions of $r$ given in~\eqref{eq:existence of CSD of Q-power monoids}, replacing $\sum_{n \in \nn}(b_{q,n}(r) - c_{q,n}(r))q^n$ by $a_q \frac{1}{\mathsf{d}(q)^{e_q}}$, and multiplying the obtained identity by $m$, one obtains that
	\begin{equation} \label{eq:aux aux}
		0 = m(b_0(r) - c_0(r)) + m\sum_{q \in Q} a_q \frac{1}{\mathsf{d}(q)^{e_q}} = m(b_0(r) - c_0(r)) + \sum_{q \in Q} n_qa_q.
	\end{equation}
	Thus, for each $t \in Q$, the equality $\gcd(\mathsf{d}(t)^{e_t}, n_t) = 1$ holds and so we can use that $\mathsf{d}(t)^{e_t} \mid m$ and also that $\mathsf{d}(t)^{e_t} \mid n_q$ for all $q \in Q$ with $q \neq t$ to infer $\mathsf{d}(t)^{e_t} \mid a_t$, from which $a_t = 0$. This implies that, for each $q \in Q$, the equality $\sum_{n \in \nn} b_{q,n}(r)q^n = \sum_{n \in \nn} c_{q,n}(r)q^n$. Hence the uniqueness of the canonical sum decomposition inside $M_q$ guarantees that $b_{q,n}(r) = c_{q,n}(r)$ for all pairs $(q,n) \in Q \times \nn$ and, therefore, $b_0(r) = c_0(r)$ by virtue of~\eqref{eq:aux aux}. Hence we have proved the uniqueness of~\eqref{eq:CSD for $Q$-power monoids}, which concludes the proof.
\end{proof}

Now that we have generalized to powerly generated monoids the canonical sum decomposition initially established in Proposition~\ref{prop:canonical decomposition} for internal sum of $p$-adic valuation Puiseux monoids, let us also extend the corresponding terminology.

\begin{defin} \label{def:CSD for Q-power monoids}
	Let $Q$ be a subset of $\qq_{> 0} \setminus \nn$ such that $2 \le |Q| < \infty$ and $\gcd(\mathsf{d}(Q)) = 1$. For each $r \in M_Q$, we call the right-hand side of~\ref{eq:CSD for $Q$-power monoids} the \emph{canonical sum decomposition} of $r$ inside $M_Q$.
	\begin{itemize}
		\item We let $c_0 \colon M_Q \to \nn_0$ be the function defined via the assignments $c_0 \colon r \mapsto c_0(r)$ for all $r \in M_Q$, where $c_0$ is as in~\eqref{eq:CSD for $Q$-power monoids}.
		\smallskip
		
		\item For each pair $(q,n) \in Q \times \nn$, we let $c_{q,n} \colon M_Q \to \nn_0$ be the function defined via the assignments $c_{q,n} \colon r \mapsto c_{q,n}(r)$ for all $r \in M_Q$, where $c_{q,n}(r)$ is as in~\eqref{eq:CSD for $Q$-power monoids}.
	\end{itemize}
\end{defin}

Keep the notation as in Definition~\ref{def:CSD for Q-power monoids}. First observe that, for any $r \in M_Q$,
\begin{equation} \label{eq:c_0(r) is the maximum integer dividing r}
	c_0(r) = \max\{n \in \nn_0 : n \mid_{M_Q} r \}.
\end{equation}
Now fix $q \in Q$ and then take $r \in M_q$. If $c_0(r) = 0$, then we can readily infer from the maximality of $c_0(r)$ that any element $s \in M_Q$ such that $s \mid_{M_Q} r$ must satisfy the following two conditions: $c_0(s) = 0$ and $s \in M_q$. Thus, the submonoids $M_q$ of $M_Q$ behave somehow like divisor-closed submonoids. Let us record these observations as the following lemma.

\begin{lem} \label{lem:aux consequence of the CSD in Q-power monoids}
	Let $Q$ be a subset of $\qq_{> 0} \setminus \nn$ such that $2 \le |Q| < \infty$ and $\gcd(\mathsf{d}(Q)) = 1$. Then the following statements hold for any $q \in Q$ and $r \in M_q$:
	\begin{enumerate}
		\item $c_0(r) = \max\{n \in \nn_0 : n \mid_{M_Q} r \}$.
		\smallskip
		
		\item If $c_0(r) = 0$, then any divisor of $r$ in $M_Q$ must belong to $M_q$.
	\end{enumerate}
\end{lem}

With Lemma~\ref{lem:aux consequence of the CSD in Q-power monoids}  and the canonical sum decomposition established in Proposition~\ref{prop:CSD for Q-power monoids} at our disposal, we can further investigate the atomic structure of powerly generated monoids. Although powerly generated monoids are not always closed under multiplication, they share various relevant atomic and factorization properties with the additive monoids $M_q$. The following result, which characterizes when a powerly generated monoid is atomic/antimatter (in terms of its internal summands $M_q$) generalizes \cite[Theorem~6.2]{GG18}.

\begin{theorem} \label{thm:atomicity of Q-power monoids}
	Let $Q$ be a subset of $\qq_{> 0} \setminus \nn$ such that $2 \le |Q| < \infty$ and $\gcd(\mathsf{d}(Q)) = 1$. Then the following statements hold.
	\begin{enumerate}
		\item $\big( \bigcup_{q \in Q} \mathcal{A}(M_q) \big) \setminus \{1\} \subseteq \mathcal{A}(M_Q)$.
		\smallskip
		
		\item $M_Q$ is atomic if and only if $M_q$ is atomic for all $q \in Q$, in which case
		\[
			\mathcal{A}(M_Q) = \bigcup_{q \in Q} \mathcal{A}(M_q).
		\]
		
		\item $M_Q$ is antimatter if and only if $M_q$ is antimatter for all $q \in Q$, in which case, $Q$ consists of unit fractions.
	\end{enumerate}
\end{theorem}

\begin{proof}
	(1) If $M_q$ is not atomic for any $q \in Q$, then $M_q$ is antimatter for all $q \in Q$ and the desired inclusion trivially holds. Thus, we assume that $M_q$ is atomic for some $q \in Q$. It suffices to fix $a \in Q$ such that $M_a$ is atomic and show that $a^k \in \mathcal{A}(M_Q)$ for every $k \in \nn$. Fix $k \in \nn$ and take $r,s \in M_Q$ such that $a^k = r+s$. By the uniqueness of the canonical sum decomposition of $a^k$ inside~$M_Q$, we infer that $c_0(a_k) = 0$. Thus $r, s \in M_a$ by virtue of Lemma~\ref{lem:aux consequence of the CSD in Q-power monoids}. Since $a^k \in \mathcal{A}(M_a)$, either $r=0$ or $s=0$. Hence $a^k \in \mathcal{A}(M_Q)$, as desired.
	\smallskip
	
	(2) For the direct implication suppose that $M_Q$ is atomic. Assume, towards a contradiction, that there exists $r \in Q$ such that $M_r$ is not atomic: in this case, $r := \frac1d \in Q$ for some $d \in \nn_{\ge 2}$. As $c_0(r) = 0$, it follows from part~(2) of Lemma~\ref{lem:aux consequence of the CSD in Q-power monoids} that, for each $q \in M \setminus \{r\}$, the only element of $M_q$ that divides $r$ in $M_Q$ is $0$. This, along with the fact that $M_Q = \sum_{q \in Q} M_q$, ensures that the only divisors of $r$ in $M_Q$ are those elements dividing $r$ in $M_q$. Thus, $r$ cannot have any factorization in $M_Q$ because $M_q$ is antimatter, contradicting the initial assumption that $M_Q$ is atomic. For the reverse implication, suppose that $M_q$ is atomic for all $q \in Q$. In this case, the set $\{q^n : n \in \nn_0\} \subseteq \mathcal{A}(M_q)$ for all $q \in Q$, and so $\{q^n : (q,n) \in Q \times \nn_0\} \subseteq \bigcup_{q \in Q} \mathcal{A}(M_q)$. Hence $M_Q$ is atomic.
	\smallskip
	
	Let us argue now that $\mathcal{A}(M_Q) = \bigcup_{q \in Q} \mathcal{A}(M_q)$ under the assumption that $M_Q$ is atomic, which is equivalent to the statement that $M_q$ is atomic for all $q \in Q$. Since $\{q^n : (q,n) \in Q \times \nn_0\}$ is a subset of $\bigcup_{q \in Q} \mathcal{A}(M_q)$ that generates the reduced monoid $M_Q$, the inclusion $\mathcal{A}(M_Q) \subseteq \bigcup_{q \in Q} \mathcal{A}(M_q)$ must hold. In light of part~(1), proving the reverse inclusion amounts to argue that $1 \in \mathcal{A}(M_Q)$. To do so, write $1 = r+s$ for some $r,s \in \mathcal{A}(M_Q)$, and then write
	\[
		r = \sum_{q \in Q} r_q \quad  \text{and} \quad s = \sum_{q \in Q} s_q,
	\]
	where $r_q, s_q \in M_q$ for all $q \in Q$. As $\gcd(\mathsf{d}(Q)) = 1$, from the equality $1 = \sum_{q \in Q} (r_q + s_q)$, we obtain that $r_q + s_q \in \zz$ for all $q \in Q$, whence we can take $t \in Q$ such that $r_t + s_t = 1$ and $r_q + s_q = 0$ for all $q \in Q \setminus \{t\}$. As $M_t$ is atomic, $1 \in \mathcal{A}(M_t)$, and so $\{r,s\} = \{r_t, s_t\} = \{0,1\}$. Hence $1 \in \mathcal{A}(M_Q)$.
	\smallskip

	(3) For the direct implication, suppose that $M_Q$ is antimatter. Assume, towards a contradiction, that $M_q$ is not antimatter for some $q \in Q$. Then $M_q$ is atomic, and the fact that $q \notin \nn$ ensures that $\mathcal{A}(M_q) = \{q^n : n \in \nn_0\}$. By the uniqueness of the canonical sum decomposition, $c_0(q) = 0$ and so it follows from Lemma~\ref{lem:aux consequence of the CSD in Q-power monoids} that the only elements dividing $q$ in $M_Q$ must belong to $M_q$. Hence $q$ must be an atom of $M_Q$, contradicting that $M_Q$ is antimatter.
	
	For the reverse implication, suppose that $M_q$ is antimatter for all $q \in Q$. Then, for any $q \in Q$, the fact that $M_q$ is antimatter implies that, for any $n \in \nn_0$, the element $q^n$ is not an atom of the monoid $M_q$, and so that $q^n \notin \mathcal{A}(M_Q$. As a result, none of the elements in the generating set $\{q^n : (q,n) \in Q \times \nn_0\}$ of $M_Q$ is an atom. Hence $M_Q$ must be antimatter.
	\smallskip
	
	Finally, observe that as $Q$ contains no integers, for each $q \in Q$, the monoid $M_q$ is antimatter if and only if $q$ is a unit fraction. Thus, we conclude that $M_Q$ is antimatter if and only if the set $Q$ consists of unit fractions.
\end{proof}

Inside the class $\{M_q : q \in \qq_{> 0}\}$, each monoid that satisfies the ACCP is an FFM: indeed, by virtue of \cite[Theorem~4.11]{CG22}, being an FFM, being a BFM, and satisfying the ACCP are equivalent conditions in the class $\{\nn_0[\alpha] : \alpha \in \cc \text{ is algebraic} \}$. Now that we have characterized when powerly generated monoids are atomic, let us characterize when these monoids are FFMs.

\begin{theorem} \label{thm:Q-power monoids with the FF property}
	Let $Q$ be a subset of $\qq_{> 0} \setminus \nn$ such that $2 \le |Q| < \infty$ and $\gcd(\mathsf{d}(Q)) = 1$. Then the monoid powerly generated by $Q$ is an FFM if and only if $q \ge 1$ for all $q \in Q$.
\end{theorem}

\begin{proof}
	For the direct implication, assume that $q < 1$ for some $q \in Q$. It is well known that, for each rational $r \in (0,1)$, the underlying additive monoid of the rational semiring $\nn_0[r]$ does not satisfy the ACCP, whence $M_q$ does not satisfy the ACCP. Let $(r_n + M_q)_{n \ge 1}$ be an ascending chain of principal ideals of $M_q$ that does not stabilize. As $M_Q$ is a reduced monoid, for each $n \in \nn$, if the inclusion $r_n + M_q \subseteq r_{n+1} + M_q$ is strict, so is the inclusion $r_n + M_Q \subseteq r_{n+1} + M_Q$. Hence $(r_n + M_Q)_{n \ge 1}$ is an ascending chain of principal ideals of $M_Q$ that does not stabilize, and so $M_Q$ does not satisfy the ACCP. Thus, $M_Q$ is not an FFM.
	\smallskip
	
	Conversely, suppose that $q \ge 1$ for every $q \in Q$. Then it follows from~\cite[Proposition~4.5]{fG19} that $M_q$ is a BFM. Thus, it follows from \cite[Theorem~4.4]{CG22} that the monoid $M_q$ is also an FFM. Finally, as $M_q$ is an FFM for all $q \in Q$, the fact that $M_Q = \sum_{q \in Q} M_q$ implies that $M_Q$ is also an FFM, which concludes our proof. 
\end{proof}

As the final result of this paper, we provide a necessary condition for a powerly generated monoid to be closed under multiplication.

\begin{prop} \label{prop:Q-power monoids that are semirings contains no unit fractions}
	Let $Q$ be a subset of $\qq_{> 0} \setminus \nn$ such that $2 \le |Q| < \infty$ and $\gcd(\mathsf{d}(Q)) = 1$. If the monoid powerly generated by $Q$ is closed under multiplication, then $Q$ cannot contain unit fractions.
\end{prop}

\begin{proof}
	Assume that $M_Q$ is closed under multiplication. Suppose, by way of contradiction, that $Q$ contains a unit fraction, namely, $r$. As $M_Q$ is closed under multiplication, for each $q \in Q \setminus \{r\}$, the product $rq$ belongs to $M_Q$ and so $q \in \mathsf{d}(r) M_Q$, which implies that $q \notin \mathcal{A}(M_Q)$. As the sets $Q \setminus \{r\}$ and $\mathcal{A}(M_Q)$ do not intersect, for each $q \in Q \setminus \{r\}$, the monoid $M_q$ is antimatter and so $q$ must be a unit fraction. Therefore, for any $q \in Q \setminus \{r\}$, the fact that $qr \in M_Q$ allows us to take coefficients $c_q, c_r \in \nn_0$ and exponents $e_q, e_r \in \nn_0$ such that
	\begin{equation} \label{eq:aux Q-power monoids that are semirings}
		qr = c_q \frac1{\mathsf{d}(q)^{e^q}} + c_r \frac1{\mathsf{d}(r)^{e_r}} + N,
	\end{equation}
	where $N \in \sum_{t \in Q \setminus \{q,r\}} M_t$. Since the coefficients $c_q$ and $c_r$ are nonzero, we can assume that $\mathsf{d}(q) \nmid c_q$ and $\mathsf{d}(r) \nmid c_r$. Now, as $\gcd(\mathsf{d}(Q)) = 1$, from the fact that the denominator of $qr$ is $\mathsf{d}(r) \mathsf{d}(q)$ we obtain that $e_q = e_r = 1$ and $N \in \nn_0$. Then the inequality $qr < 1$ guarantees that $N = 0$. After multiplying \eqref{eq:aux Q-power monoids that are semirings} by $\mathsf{d}(q) \mathsf{d}(r)$, we obtain the equality $1 = c_q\mathsf{d}(r) + c_r \mathsf{d}(q)$, which is clearly a contradiction. Hence we conclude that~$Q$ cannot contain unit fractions.
\end{proof}

The necessary condition given in Proposition~\ref{prop:Q-power monoids that are semirings contains no unit fractions} cannot be used to characterize powerly generated monoids that are closed under multiplication. The following example illustrates this observation.

\begin{ex}
	Let $p_1$ and $p_2$ be two primes such that $p_1 < p_2$. Now set $Q := \{q_1, q_2\}$, where $q_1 := \frac{p_2 - p_1}{p_1}$ and $q_2 := \frac{p_1}{p_2}$. Thus, $Q$ is a non-singleton subset of $\qq_{> 0}$ with $\gcd(\mathsf{d}(Q)) = \gcd(p_1, p_2) = 1$, and $Q$ does not contain unit fractions. Now consider the monoid powerly generated by $Q$:
	\[
		M_Q := \big\langle q_1^n, q_2^n : n \in \nn_0 \big\rangle.
	\]
	As neither $q_1$ nor $q_2$ is a unit fraction, the monoids $M_{q_1}$ and $M_{q_2}$ are both atomic. Hence it follows from part~(2) of Theorem~\ref{thm:atomicity of Q-power monoids} that $M_Q$ is atomic and also that $1 \in \mathcal{A}(M_Q)$. Since
	\[
		q_1 q_2 + q_2 = \left(\frac{p_2 - p_1}{p_1}\right)\left(\frac{p_1}{p_2}\right) + \frac{p_1}{p_2} = 1,
	\]
	from the fact that $1$ is an atom of $M_Q$ we obtain that $q_1 q_2 \notin M_Q$. Hence $M_Q$ is not closed under the standard multiplication.
\end{ex}

We have seen in Theorem~ \ref{thm:multiplicatively cyclic are LFF iff are atomic} that, for any $q \in \qq_{> 0}$, the monoid $M_q$ is an LFFM if and only if $q$ is not a unit fraction. Based on this, we conclude with the following open question.

\begin{question}
	Let $Q$ be a subset of $\qq_{> 0} \setminus \nn$ such that $2 \le |Q| < \infty$ and $\gcd(\mathsf{d}(Q)) = 1$. Can we characterize in terms of $Q$ when the monoid $M_Q$ powerly generated by $Q$ is an LFFM?
\end{question}

\bigskip
\section*{Acknowledgments}

While working on this paper, the second author was supported first by the UC Berkeley Chancellor Fellowship and then by the NSF (under the awards DMS-1903069 and DMS-2213323). Also, the fourth author was first funded by the UF Mathematics Department Fellowship and then by the UC Presidential Fellowship.

\bigskip


\begin{thebibliography}{20}

	\bibitem{ABLST23} K. Ajran, J. Bringas, B. Li, E. Singer, and M. Tirador, \emph{Factorization in additive monoids of evaluation polynomial semirings}, Comm. Algebra \textbf{51} (2023) 4347--4362.
	
	\bibitem{ABP21} S. Albizu-Campos, J. Bringas, and H. Polo, \emph{On the atomic structure of exponential Puiseux monoids and semirings}, Comm. Algebra \textbf{49} (2021) 850--863.
	
	\bibitem{ACHP07} J. Amos, S.~T. Chapman, N. Hine, and J. Paixao, \emph{Sets of lengths do not characterize numerical monoids}, Integers \textbf{7} (2007) A50.
	
	\bibitem{AAZ90} D.~D. Anderson, D.~F. Anderson, and M. Zafrullah, \emph{Factorization in integral domains}, J. Pure Appl. Algebra \textbf{69} (1990) 1--19.
	
	
	\bibitem{AQ97} D.~D. Anderson, R.~O. Quintero, \emph{Some generalizations of GCD-domains}. In: Factorization in Integral Domains (Ed. D. D. Anderson) pp. 189--195, Lectures in Pure and Applied Mathematics, Marcel Dekker, New York 1997.
	
		
		
		
		
		
	\bibitem{BG21} M. Bras-Amor\'os and M.~Gotti, \emph{Atomicity and density of Puiseux monoids}, Comm. Algebra \textbf{49} (2021) 1560--1570.
		
	
		
	\bibitem{CGLM11} S.~T. Chapman, P.~A. Garc\'ia-S\'anchez, D. Llena, and J. Marshall, \emph{Elements in a numerical semigroup with factorizations with the same length}, Canadian Math. Bull. \textbf{54} (2011) 39--43.
			
	\bibitem{CGG20} S. T. Chapman, F. Gotti, and M. Gotti, \emph{Factorization invariants of Puiseux monoids generated by geometric sequences}, Comm. Algebra \textbf{48} (2020) 380--396.
	
	\bibitem{CGG21} S.~T. Chapman, F. Gotti, and M. Gotti, \emph{When is a Puiseux monoid atomic?}, Amer. Math. Monthly \textbf{128} (2021) 302--321.
		
	\bibitem{CG24} S. T. Chapman and M. Gotti, \emph{Atomicity of positive monoids}, Quaestiones Mathematicae \textbf{47} (2024) 2177--2193.

	\bibitem{CJMM25} S.~T. Chapman, J. Jang, J. Mao, and S. Mao, \emph{On the set of Betti elements of a Puiseux monoid}, Bull. Aust. Math. Soc. \textbf{111} (2025) 80--92.
	
	\bibitem{CP23} S. T. Chapman and H. Polo, \emph{Arithmetic of additively reduced monoid semidomains}, Semigroup Forum \textbf{107} (2023) 40--59.
	
	\bibitem{pC68} P. M. Cohn, \emph{Bezout rings and their subrings}, Proc. Cambridge Philos. Soc. \textbf{64} (1968) 251--264.

    \bibitem{CG22} J. Correa-Morris and F. Gotti, \emph{On the additive structure of algebraic valuations of polynomial semirings}, J. Pure Appl. Algebra \textbf{226} (2022) 107104.
            
	\bibitem{CDM99} J. Coykendall, D.~E. Dobbs, and B. Mullins, \emph{On integral domains with no atoms}, Communications in Algebra 27 (1999) 5813--5831.
		
	\bibitem{CG19} J. Coykendall and F. Gotti, \emph{On the atomicity of monoid algebras}, J. Algebra \textbf{539} (2019) 138--151.

    \bibitem{lF70} L. Fuchs, \emph{Infinite Abelian Groups I}, Academic Press, 1970.
            
	\bibitem{GR09} P.~A. Garc\'ia-S\'anchez and J.~C. Rosales: \emph{Numerical Semigroups}, Developments in Mathematics, 20, Springer-Verlag, New York, 2009.
		
	\bibitem{GG25} A. Geroldinger and F. Gotti, \emph{On monoid algebras having every nonempty subset of $\mathbb{N}_{\ge 2}$ as a length set}, Mediterr. J. Mathematics \textbf{22} (2025). DOI: https://doi.org/10.1007/s00009-025-02835-0. Preprint on arXiv: https://arxiv.org/abs/2404.11494.
			
	\bibitem{GH06} A.~Geroldinger and F.~Halter-Koch: \emph{Non-Unique Factorizations: Algebraic, Combinatorial and Analytic Theory}, Pure and Applied Mathematics, vol.~278, Chapman \& Hall/CRC, Boca Raton, 2006.
	
	\bibitem{GS18} A. Geroldinger and W. Schmid, \emph{A realization theorem for sets of lengths in numerical monoids}, Forum Mathematicum \textbf{30} (2018) 1111--1118.
	
	\bibitem{GZ21} A. Geroldinger and Q. Zhong: \emph{A characterization of length-factorial Krull monoids}, New York J. Math. \textbf{27} (2021) 1347--1374.
		
	\bibitem{rG84} R.~Gilmer: \emph{Commutative Semigroup Rings}, Chicago Lectures in Mathematics, The University of Chicago Press, London, 1984.
		
	\bibitem{GGP25} V. Gonzalez, F. Gotti, and I. Panpaliya, \emph{On the ascent of almost and quasi-atomicity to monoid semidomains}. Submitted. Preprint on arXiv: https://arxiv.org/abs/2501.04990.
			
	\bibitem{fG19} F. Gotti, \emph{Increasing positive monoids of ordered fields are FF-monoids}, J. Algebra \textbf{518} (2019) 40--56.

    \bibitem{fG20} F. Gotti, \emph{Irreducibility and factorizations in monoid rings}. In: Numerical Semigroups (Eds. V. Barucci, S. T. Chapman, M. D'Anna, and R. Fr\"oberg) pp. 129--139. Springer INdAM Series, vol \textbf{40}. Springer, Cham.
    
	\bibitem{fG22} F. Gotti: \emph{On semigroup algebras with rational exponents}, Comm. Algebra \textbf{50} (2022) 3--18.
	
	\bibitem{fG17} F. Gotti, \emph{On the atomic structure of Puiseux monoids}, J. Algebra Appl. \textbf{16} (2017) 1750126.
	
	\bibitem{fG18} F. Gotti, \emph{Puiseux monoids and transfer homomorphisms}, J. Algebra \textbf{516} (2018) 95--114.
	
	\bibitem{fG19a} F. Gotti, \emph{Systems of sets of lengths of Puiseux monoids}, J. Pure Appl. Algebra
	\textbf{223} (2019) 1856--1868.
	
	\bibitem{GG18} F. Gotti and M. Gotti, \emph{Atomicity and boundedness of monotone Puiseux monoids}, Semigroup Forum \textbf{96} (2018) 536--552.
	
	\bibitem{GO20} F. Gotti and C. O'Neill, \emph{On the elasticity of Puiseux monoids}, Commut. Algebra \textbf{12} (2020) 319--331.
	
	\bibitem{GR25} F. Gotti and H. Rabinovitz, \emph{On the ascent of atomicity to monoid algebras}, J. Algebra \textbf{663} (2025) 857--881. 

    \bibitem{aG74} A.~Grams, \emph{Atomic rings and the ascending chain condition for principal ideals}, Math. Proc. Cambridge Philos. Soc. \textbf{75} (1974) 321--329.
            
	\bibitem{pG01} P.~A. Grillet, \emph{Commutative Semigroups}, Advances in Mathematics, vol.~2, Kluwer Academic Publishers, Boston, 2001.
	
	\bibitem{fH92} F. Halter-Koch, \emph{Finiteness theorems for factorizations}, Semigroup Forum \textbf{44} (1992) 112--117.
	
	\bibitem{JKK24} H. Jiang, S. Kanungo, and H. Kim, \emph{On a weaker notion of the finite factorization property}, Commun. Korean Math. Soc. \textbf{39} (2024) 313--329.
	
    \bibitem{hP23} H. Polo, \emph{Factorization invariants of the additive structure of exponential Puiseux semirings}, J. Algebra Appl. \textbf{22} (2023) 2350077.

	\bibitem{hP20} H. Polo, \emph{On the sets of lengths of Puiseux monoids generated by multiple geometric sequences}, Commun. Korean Math. Soc. (2020) \textbf{35} 1057--1073.

    \bibitem{aT17} A. Tripathi, \emph{Formulae for the Frobenius number in three variables}, J. Number Theory \textbf{170} (2017) 368--389.
\end{thebibliography}
\end{document}